\documentclass[11pt,twoside,reqno]{amsart}
\usepackage{amsmath, amsfonts, amsthm, amssymb,amscd, graphicx, amscd}

\usepackage{movie15}

\usepackage{float,epsf}
\usepackage[english]{babel}
\usepackage{enumerate}
\usepackage{tikz}
\usepackage{hyperref}

\usepackage{srcltx}

\usepackage{geometry}
\usepackage{verbatim}
\usepackage{mathrsfs}

\graphicspath{ {11./images/} }

\newtheorem{theorem}{Theorem}[section]
\newtheorem{lemma}[theorem]{Lemma}

\newtheorem{definition}[theorem]{Definition}

\newtheorem{proposition}[theorem]{Proposition}

\newtheorem{remark}[theorem]{Remark}
\newtheorem{example}[theorem]{Example}

\newtheorem{question}[theorem]{Question}

\newtheorem*{proposition*}{Proposition}

\numberwithin{equation}{section}
\numberwithin{figure}{section}

\DeclareMathOperator{\sech}{sech}

\def\intave#1{\int_{#1}\hbox{\llap{$\raise2.3pt\hbox{\vrule
height.9pt width7pt}\phantom{\scriptstyle{#1}}\mkern-2mu$}}}

\begin{document}
\title{Some geometric inequalities related to Liouville equation}

\author{Changfeng Gui}
\address{Changfeng Gui, The University of Texas at San Antonio, TX, USA.
}
\email{changfeng.gui@utsa.edu}
\author{Qinfeng Li}
\address{Qinfeng Li, Hunan University, Changsha, Hunan, China.
}
\email{liqinfeng1989@gmail.com}

\begin{abstract}
In this paper, we prove that if $u$ is a solution to the Liouville equation \begin{align}
\label{scalliouville}
\Delta u+e^{2u} =0 \quad \mbox{in $\mathbb{R}^2$,}
\end{align}then the diameter of $\mathbb{R}^2$ under the conformal metric $g=e^{2u}\delta$ is bounded below by $\pi$. Here $\delta$ is the Euclidean metric in $\mathbb{R}^2$. Moreover, we explicitly construct a family of solutions to \eqref{scalliouville} such that the corresponding diameters of $\mathbb{R}^2$ range over $[\pi,2\pi)$.  

We also discuss supersolutions to \eqref{scalliouville}. We show that if $u$ is a supersolution to \eqref{scalliouville} and $\int_{\mathbb{R}^2} e^{2u} dx<\infty$, then the diameter of $\mathbb{R}^2$ under the metric $e^{2u}\delta$ is less than or equal to $2\pi$.

For radial supersolutions to \eqref{scalliouville}, we use both analytical and geometric approaches to prove some inequalities involving conformal lengths and areas of disks in $\mathbb{R}^2$. We also discuss the connection of the above results with the sphere covering inequality in the case of Gaussian curvature bounded below by $1$.

Higher dimensional generalizations are also discussed.
\end{abstract}
\maketitle

\vskip0.2in



\section{Introduction}

In this paper, we study some geometric properties of the conformally flat Riemannian manifold $(M,g)=(\mathbb{R}^2, e^{2u}\delta)$, where $u$ being considered is either a solution or a supersolution to \eqref{scalliouville} and $\delta$ is denoted as the Euclidean metric. By $u$ being a supersolution to \eqref{scalliouville}, we mean that 
\begin{align}
    \label{ricci1}
\Delta u+e^{2u} \le 0   \quad \mbox{in $\mathbb{R}^2$.}
\end{align} 
The research in this paper consists of the following parts.

\subsection{Diameter estimates and examples for solutions to \eqref{scalliouville}}
It is well known that if $u$ is a solution to \eqref{scalliouville}, then either $vol_g(\mathbb{R}^2)=4\pi$ or $vol_g(\mathbb{R}^2)=\infty$, where $g=e^{2u}\delta$ and $vol_g$ denotes the volume function under the meric $g$. In the former case, $u$ must be a radial solution up to translations. This was first proved in the seminal paper by Chen-Li\cite{CL91} using the  moving plane method, and later by Chou-Wan\cite{CW94} using Liouville's formula (see \cite{Liouville}): any entire solution $u$ to \eqref{scalliouville} must be of the form
\begin{align}
    \label{complexrepresentation}
u(x,y)=\ln \frac{2|f'(z)|}{1+|f(z)|^2},  
\end{align}
where $f$ is a meromorphic function in the complex plane, $f$ has only simple poles and $f'$ is nonvanishing. Such $f$ is called the \textit{developing function} of solution $u$ to \eqref{scalliouville}.   In particular,  if $f(z)=\lambda z$ for some  positive constant $\lambda$, i.e., 
\begin{align}
\label{bubble}
u(x,y)=\ln \frac{2 \lambda}{1+\lambda^2 |z|^2}, 
\end{align}
then we call them the \textit{standard bubble solutions}, for which  $g= e^{2u} \delta$ represent the scaled metrics of the standard metric on the unit sphere after identifying the sphere with the complex plane
compactified at infinity.

Motivated by the above results, we are interested in studying the diameter of $\mathbb{R}^2$ under the conformal metric $g=e^{2u}\delta$ for $u$ being a solution to \eqref{scalliouville}. Throughout this paper we use $diam_g(\mathbb{R}^2)$ to denote the diameter of $\mathbb{R}^2$ under metric $g$.

We first prove the following diameter lower bound estimate in Section 2.
\begin{theorem}
\label{lowerbounddiameterfor2D}
Let $u$ be a solution to \eqref{scalliouville} in $\mathbb{R}^2$, then $diam(\mathbb{R}^2) \ge \pi$ under the metric $g=e^{2u}\delta$.
\end{theorem}
We also prove a stronger version, see Proposition \ref{manypoints} in the section below, which indicates that there are uncountably many pairs of points such that the conformal distances between any of the pairs are bigger than or equal to $\pi$. 

Next, we  construct a family of solutions to \eqref{scalliouville} such that the corresponding diameters can be attained at any value in $[\pi,2\pi)$.. 

More precisely, we show:
\begin{proposition}
\label{utdiamter}
Let $u_t$ be a family of functions given as  \begin{align}
\label{u_t}
    u_t(x,y)=\frac{2e^x}{1+t^2+2te^x \cos y+e^{2x}},
\end{align}
then for each $t \ge 0$, $u_t(x,y)$ solves \eqref{scalliouville}. Moreover, $diam_g(\mathbb{R}^2) = \pi+2\tan^{-1}(t)$, where $g=e^{2u_t}\delta$. 
\end{proposition}

One can see from the above proposition that if $t$ ranges over $[0,\infty)$, then the diameters of $\mathbb{R}^2$ corresponding to $u_t$ ranges over $[\pi,2\pi)$. This is somehow interesting, since unlike that the range of conformal volume is discrete regarding solutions to \eqref{scalliouville}, it turns out that the range of conformal diameter contain an interval.


\begin{remark}
Actually in our forthcoming paper \cite{EGLX}, we can show if a solution $u$ is bounded from above, then up to translation, rotation and scaling, either $u$ is radial, or $u$ is given by \eqref{u_t}. Hence we essentially proved when $u$ has an upper bound, the range of diameter of $\mathbb{R}^2$ under $e^{2u}\delta$ is $[\pi,2\pi)$.
\end{remark}

We also construct a solution in Example \ref{2pi}, where we show that the corresponding conformal diameter can be greater than or equal to $2\pi$. 

The difficulty of exactly computing the conformal diameters lies in two aspects. First, generally given two points in $\mathbb{R}^2$ and a solution $u$, there is no standard way to compute the conformal distance between the two points under the metric $g=e^{2u}\delta$, since there are infinitely many paths connecting them. Second, generally it is not clear how to choose pairs of points such that their conformal distances are approximating the conformal diameter of $\mathbb{R}^2$. The method we use to prove Proposition \ref{utdiamter} and illustrate Example \ref{2pi} is by finding the link between the conformal metrics with respect to solutions to \eqref{scalliouville} and the standard metric on sphere, and we also employ some complex analysis ideas to carefully proceed the argument.
\hfill\\
\subsection{Diameter estimates for general supersolutions}
All the above results are obtained using complex function theory. However, from a geometric point of view, \eqref{scalliouville} is equivalent to $K =1$, where $K$ is the Gaussian curvature of $(\mathbb{R}^2,g)$, where $g=e^{2u}\delta$. Similarly, when considering supersolutions, \eqref{ricci1} is equivalent to $K \ge 1$. Note that \eqref{ricci1} is also equivalent to $Ric_g \ge g$, where $Ric_g$ is the Ricci curvature of $(\mathbb{R}^2, g)$. This naturally reminds us of Myer's Theorem in Riemannian geometry.

Recall that Meyer's Theorem says that if $(M,g)$ is a complete $n$-dimensional manifold such that $Ric_g \ge (n-1)g$, then $diam_g(M) \le \pi$. This is not true for incomplete manifold, and Proposition \ref{utdiamter} actually serves as a counterexample since in such case $Ric_g=(2-1)g$ while $diam_g(\mathbb{R}^2)>\pi$ if $t>0$.

Even though nothing can be said in general on diameter bound for incomplete Riemannian manifolds, surprisingly, we can prove that if $g$ is a globally conformally flat metric in $\mathbb{R}^2$ with $Ric_g \ge g$ and $vol_g(\mathbb{R}^2)<\infty$, then $diam_g(\mathbb{R}^2) \le 2\pi$. We state this result in the following theorem in PDE language:
\begin{theorem}
\label{weakbounddiamter}
Let $u$ satisfy \eqref{ricci1}. If we also assume that $\int_{\mathbb{R}^2}e^{2u} dx<\infty$, then $diam_g(\mathbb{R}^2) \le 2\pi$, where $g=e^{2u}\delta$.
\end{theorem}
This theorem will be proved in section 4, and the argument of our proof does apply Myer's Theorem in some situation while borrows the idea of proof of Hopf-Rinow Theorem. A key observation is Proposition \ref{diametercontrolinfty}, where no finiteness of volume assumption is needed.

We note that given the assumptions in Theorem \ref{weakbounddiamter}, we do not know whether or not the $2\pi$ upper bound for the conformal diameter is sharp. It is interesting to find a supersolution $u$ satisfying all assumptions in Theorem \ref{weakbounddiamter} such that the conformal diameter is strictly between $\pi$ and $2\pi$, or to prove that the upper bound should be $\pi$. So far we haven't had an answer yet.
\hfill\\
\subsection{Geometric inequalities related to radial supersolutions.}
We also study radial supersolutions to \eqref{scalliouville}. This is essentially an ordinary differential inequality problem. Let us define the conformal perimeter and area of balls in $\mathbb{R}^2$ by $l(r)=\int_{\partial B_r}e^{u}ds$ and $A(r)=\int_{B_r}e^{2u} dx$, where $B_r$ is the ball of radius $r$ in $\mathbb{R}^2$ centered at the origin and $ds$ is the length element. It turns out that many inequalities involving $l(r)$ and $A(r)$ can be derived. We prove:

\begin{theorem}
\label{maintheoremonradialcase}
Let $u$ be a radially symmetric function satisfying \eqref{ricci1}, and $g=e^{2u}\delta$ where $\delta$ is the Euclidean metric, then we have 
\begin{align}
    \label{volumeradial}
vol_g(\mathbb{R}^2) \le 4\pi    
\end{align} and \begin{align}
    \label{diamradial}
diam_g(\mathbb{R}^2) \le \pi.
\end{align} Furthermore, there exists $r_0>0$ such that $l$ is increasing when $r<r_0$ and $l$ is decreasing when $r>r_0$. Moreover, 
\begin{align}
    \label{yanxing3}
\max_{r>0}l(r)=l(r_0)\le 2\pi,    
\end{align}
and 
\begin{align}
\label{jiandanyouyong}
    \lim_{r\rightarrow \infty} l(r)=0.
\end{align}
Also, for any $r>0$, \begin{align}
\label{leibi}
A(A_{\infty}-A) \le l^2 \le 4\pi A -A^2,
\end{align}where $A_{\infty}=\int_{\mathbb{R}^2} e^{2u} dx$.
In addition, let $R=R(r)=\int_0^re^{u(\rho)}d\rho$, then 
\begin{align}
    \label{boundsforr0}
R(r_0) \le \frac{\pi}{2},  
\end{align}
\begin{align}
    \label{busanhuang}
A_{\infty}(1-\cos R) \le 2A,
\end{align}
and
\begin{align}
\label{hmm}
    \frac{A}{l}\ge \frac{1-\cos R}{\sin R}.
\end{align}Moreover, if $r \le r_0$, we also have 
\begin{align}
\label{jiefaqian}
    \frac{A}{l} \le \sin R.
\end{align}
In particular, if $r \le r_0$, then \begin{align}
\label{jiefa}
    A(r) \le l(r),
\end{align}and if $R \ge \frac{\pi}{2}$, then 
\begin{align}
\label{bushanhuangtuichu}
    A \ge \frac{A_{\infty}}{2}.
\end{align}
\end{theorem}

Note that when $u$ is a radial solution for the equality case \eqref{scalliouville}, then $(\mathbb{R}^2, e^{2u}\delta)$ becomes the standard unit sphere minus a point, and $l(r)$ then corresponds to the length of the latitudes and $A(r)$ corresponds to the area of the spherical caps. Then one can see geometrically that all the inequalities \eqref{volumeradial}-\eqref{hmm} become equalities.

Let us make some comments of the proofs of Theorem \ref{maintheoremonradialcase}. \eqref{volumeradial} can be derived from either of the two inequalities in \eqref{leibi}. \eqref{diamradial} is proved using the idea of proof of Myer's Theorem and Proposition \ref{weakbounddiamter}. \eqref{yanxing3}-\eqref{jiandanyouyong}, together with the second inequality in \eqref{leibi} are proved analytically, while the rest of the inequalites are proved using geometric argument. In particular, the first inequality in \eqref{leibi}, \eqref{boundsforr0} and \eqref{jiefaqian} are all obtained by exploiting the Heintze-Karcher inequality (see \cite{HK78}), which gives a control of the Jocobian of the exponential map starting from the boundary of domain in Riemannian manifold with strictly positive Ricci curvature lower bound. \eqref{hmm} is derived from Bishop-Gromov inequality, while \eqref{busanhuang} is a consequence of \eqref{leibi} and \eqref{hmm}. \eqref{jiefa} is a consequence of \eqref{jiefaqian}, and it can also be derived from Alexandrov inequality. All of these are presented in section 5.

We also discuss some higher dimensional generalizations for radial cases in section 6. 

\hfill\\
\subsection{Connection with sphere covering inequality}
Our study of supersolutions to the Liouville equation \eqref{scalliouville} is also motivated by the famous sphere covering inequality recently discovered by Gui-Moradifam\cite{GM}, which is a very powerful inequality that has been successfully applied to solve various problems on symmetry and uniqueness properties of solutions of semilinear elliptic equations with exponential nonlinearity in $\mathbb{R}^2$. In particular,
it was applied to prove a longstanding conjecture of Chang-Yang\cite{CY87} concerning
the best constant in Moser-Trudinger-type inequalities, see \cite{GM}, and has led to several symmetry and uniqueness results for mean field equations, Onsager vortices, SinhGordon equation, cosmic string equation, Toda systems, and rigidity of Hawking mass in general relativity. References include but not limited to \cite{Du}, \cite{25}, \cite{26}, \cite{27}, \cite{GM}, \cite{29}, \cite{30}, \cite{31}, \cite{39}, \cite{41}, etc.

A simple version of the sphere covering inequality states as follows:
\begin{theorem}(\cite[Theorem 1.1]{GHM})
\label{spherecovering}
Let $u_1$ be a smooth function defined in a simply connected domain $\Omega_0$ such that 
\begin{align}
    \label{subsolution}
\Delta u_1+e^{2u_1} \ge 0\, \mbox{in $\Omega_0$},\quad \int_{\Omega_0}e^{2u_1} \le 4\pi. 
\end{align}Let $u_2$ be another smooth function defined in $\Omega_0$ and suppose that in a subdomain $\Omega\subset \Omega_0$ such that
\begin{align}
    \Delta u_2+e^{2u_2}\ge \Delta u_1+e^{2u_1} \, \mbox{in $\Omega$}, \, u_2>u_1 \, \mbox{in $\Omega$ and} \, u_2=u_1 \, \mbox{on $\partial \Omega$.}
\end{align}Then 
\begin{align}
    \int_{\Omega}e^{2u_1}dx +\int_{\Omega}e^{2u_2}dx \ge 4\pi.
\end{align}
\end{theorem}
This Theorem  was first proved by Gui-Morodifam\cite{GM}, and later proved by Gui-Hang-Morodifam\cite{GHM} in a simpler but more intrinsic way. However, both proofs use the crucial assumption \eqref{subsolution}, since \eqref{subsolution} means that $(\Omega, e^{2u_1}\delta)$ is a Riemannain surface with $K\le 1$, where $K$ is the Gaussian curvature. Only with this assumption, the following Alexander-Bol inequality can be applied: 
\begin{align}
    \label{bol}
l^2(\partial \Omega) \ge 4\pi A(\Omega)-A^2(\Omega),    
\end{align}where $l(\partial \Omega)$ is the conformal length of $\partial \Omega$ under the metric $e^{2u_1}\delta$, and $A(\Omega)$ is the conformal area. The inequality \eqref{bol} is essential in both proofs of Theorem \ref{spherecovering} in \cite{GM} and \cite{GHM}. 

The main step in proving Theorem \ref{spherecovering} in \cite{GHM} is the following theorem:

\begin{theorem} (\cite[Theorem 1.4]{GHM})
\label{ghm1.4}
Let $(M,g)$ be a simply connected Riemannian surface with $\mu(M) \le 4\pi$ and $K \le 1$, where $\mu$ is the measure of $(M,g)$ and $K$ is the Gaussian curvature. Let $\Omega$ be a domain with compact closure and nonempty boundary, and $\lambda$ is a constant. Then if $u \in C^2(\overline{\Omega})$ satisfying
\begin{align}
\label{1.4}
  \begin{cases}
-\Delta_g u+1 \le \lambda e^{2u}, \quad  u>0 \quad &\mbox{in $\Omega$}\\
u=0 \quad &\mbox{on $\partial \Omega$}
\end{cases}  
\end{align}
Then \begin{align}
\label{xinren1}
    4\pi \int_{\Omega} e^{2u} d\mu -\lambda \left(\int_{\Omega} e^{2u} d\mu \right)^2 \le 4 \pi \mu(\Omega)-\mu^2(\Omega).
\end{align}
In particular, if $\lambda \in (0,1]$, then
\begin{align}
\label{linshiruji}
    \int_{\Omega} e^{2u} d\mu+\mu(\Omega) \ge \frac{4\pi}{\lambda}.
\end{align}
\end{theorem}

For the case when Gaussian curvature $K \ge 1$, the above result is no longer true. Actually the Alexander-Bol inequality goes in the complete opposite direction, because of the second inequality in \eqref{leibi}. However, we can still say something. The following are the results we proved:

First, we prove the following counterpart of Theorem \ref{ghm1.4} in the case when $M$ is a closed surface and $K\ge 1$. In fact, we prove a slightly more general version.
\begin{theorem}
\label{genghm}
Let $(M,g)$ be a closed Riemannian surface with $K \ge 1$, where $K$ is the Gaussian curvature. Let $\mu$ be the volume measure of $(M,g)$ and $\Omega$ be a domain with nonempty boundary. If $u$ satisfies 
\begin{align}
\label{rubin2''}
  \begin{cases}
-\Delta_g u+1 \le h(u), \quad  u>0 \quad &\mbox{in $\Omega$}\\
u=0 \quad &\mbox{on $\partial \Omega$}
\end{cases} 
\end{align}where $h(t)$ is a nonnnegative function such that \begin{align}
\label{mc1''}
   0 \le  h'(t) \le 2 h(t),
\end{align}
then \begin{align}
\label{guijia1'}
    \mu(M) \int_{\Omega}h(u) d\mu -\left(\int_{\Omega} h(u) d\mu\right)^2 \le h(0) \mu(\Omega) \mu (M\setminus \Omega).
\end{align}
Moreover, if $h_0 \in (0,1]$, then 
\begin{align}
\label{guijia2'}
    \int_{\Omega}h(u) d\mu +h(0)\mu(\Omega) \ge \mu(M).
\end{align}
\end{theorem}In particular, if $h(u)=\lambda e^{2u}$, then \eqref{guijia1'} and \eqref{guijia2'} become the following forms similar to \eqref{xinren1} and \eqref{linshiruji}:
\begin{align}
\label{guijia1''}
    \mu(M) \int_{\Omega}e^{2u} d\mu -\lambda \left(\int_{\Omega} e^{2u} d\mu\right)^2 \le \mu(M)\mu(\Omega)- \mu^2(\Omega),
\end{align}and
\begin{align}
    \label{guijia2''}
\int_{\Omega}e^{2u} d\mu +\mu(\Omega) \ge \frac{\mu(M)}{\lambda}.    
\end{align}
Note that when $K=1$, then $\mu(M)=4\pi$ and hence \eqref{guijia2''} recovers \eqref{linshiruji}. If $K >1$, then $\int_{\Omega}e^{2u} d\mu +\mu(\Omega)$ has a smaller lower bound, since $\mu(M)<4\pi$.
\hfill\\

The proof of Theorem \ref{genghm} is motivated by that of Theorem \ref{ghm1.4}, and the new ingredient is the application of L\'evy-Gromov inequality instead of Alexander-Bol ineqaulity. Similarly as in \cite{GHM}, we also prove the dual form of Theorem \ref{genghm}, see Theorem \ref{genghm'} in section 7.

The closedness assumption of $M$ is crucial in Theorem \ref{genghm}, since it is crucial in either of the two proofs of L\'evy-Gromov inequality so far we have known, see \cite{Gromov} for the orginal proof by Gromov and \cite{Bayle} by Bayle for a different proof. As one can see from example \ref{noncomplete} and the comment right after, even for radial supersolution $u$, one cannot expect that $(\mathbb{R}^2, e^{2u}\delta)$ can be completed as a closed Riemannian surface. Hence a sphere covering inequality for the case $K\ge 1$ cannot be directly obtained from Theorem \ref{genghm}.

However, by the first inequality of \eqref{leibi} obtained in Theorem \ref{maintheoremonradialcase} and by a similar argument in the proof of Theorem \ref{genghm}, the following counterpart to Theorem \ref{spherecovering} for radial functions is obtained.
\begin{theorem}
\label{radialspherecovering}
Let $u_1$ be a radially symmetric function such that\begin{align*}
    \Delta u_1+e^{2u_1} \le 0\quad \mbox{in $\mathbb{R}^2$.}
\end{align*}Let $u_2$ be another radially symmetric function defined in $\mathbb{R}^2$. If for some disk $B_r$ we have
\begin{align}
    \Delta u_2+e^{2u_2}\ge \Delta u_1+e^{2u_1} \, \mbox{in $B_r$}, \, u_2>u_1 \, \mbox{in $B_r$ and} \, u_2=u_1 \, \mbox{on $\partial B_r$,}
\end{align}then
\begin{align}
    \int_{B_r} e^{2u_1}+\int_{B_r}e^{2u_2} \ge \int_{\mathbb{R}^2} e^{2u_1} dx.
\end{align}The equality holds if and only if $(B_r, e^{2u_1}\delta)$ and $(B_r,e^{2u_2}\delta)$ are two complementary spherical caps on the unit sphere. 
\end{theorem}
Similarly, we state the dual form for Theorem \ref{radialspherecovering}.
\begin{theorem}
\label{radialspherecovering'}
Let $u_1$ be a radially symmetric function such that\begin{align*}
    \Delta u_1+e^{2u_1} \le 0\quad \mbox{in $\mathbb{R}^2$.}
\end{align*}Let $u_2$ be another radially symmetric function defined in $\mathbb{R}^2$. If for some disk $B_r$ we have
\begin{align}
    \Delta u_2+e^{2u_2}\le \Delta u_1+e^{2u_1} \, \mbox{in $B_r$}, \, u_2<u_1 \, \mbox{in $B_r$ and} \, u_2=u_1 \, \mbox{on $\partial B_r$,}
\end{align}then
\begin{align}
    \int_{B_r} e^{2u_1}+\int_{B_r}e^{2u_2} \ge \int_{\mathbb{R}^2} e^{2u_1} dx.
\end{align}Moreover, the equality holds if and only if $(B_r, e^{2u_1}\delta)$ and $(B_r,e^{2u_2}\delta)$ are two complementary spherical caps on the unit sphere. 
\end{theorem}

These results are proved in section 7.

We conjecture that the above theorem holds for nonradial solutions and for general smooth domains, but so far we cannot validate this conjecture. In section 8, we list a series of important unsolved problems related to the results in this paper for future research. 

To the end, we remark that one of the essentials in the proof of L\'evy-Gromov isoperimetric inequality requires the diameter bounded above by $\pi$, which is guaranteed by Myer's Theorem on complete Riemannian manifolds with $Ric_g \ge (n-1)g$. For $n=2$, this is equivalent to $K \ge 1$. Hence in order to develop new sphere covering inequality related to incomplete globally conformally flat surface with $K \ge 1$, the first step is to try to prove some diameter bound. This has been another motivation for us to study diameter estimates for solutions and supersolutions to the Liouville equation \eqref{scalliouville}.

\section{Diameter lower bound for solutions to \eqref{scalliouville}}
In this section, we first prove Theorem \ref{lowerbounddiameterfor2D}. Before proving this, we state the following well-known lemma which is essentially proved in \cite{CW94}. Here we include a proof for readers' convenience. 

\begin{lemma}
Let $f(z)$ be the developing function for the solution $u$ to \eqref{scalliouville}, then $f$ is either a Mobi\"us transform or transcendental meromorphic. For the former case, $u$ is radially symmetric up to translations. 
\end{lemma}
\begin{proof}
If $f$ is rational, then write $f=\frac{P}{Q}$ where both $P$ and $Q$ are polynomials over $\mathbb{C}$. We assume $P$ and $Q$ have no common factors. Since $f$ has at most simple poles, $Q$ has distinct simple roots. We claim that 
\begin{align}
\label{claim}
    P'Q-PQ'\equiv C \ne 0.
\end{align}
Indeed, $f'=\frac{P'Q-PQ'}{Q^2}$. If $P'Q-PQ'$ and $Q^2$ have common factors, say $z-z_0$, then since $Q(z_0)=0$, $(P'Q-PQ')(z_0)=0$, and thus $P(z_0)Q'(z_0)=0$. Since $Q$ has simple roots, $Q'(z_0)\ne 0$, and thus $P(z_0)=0$. This contradicts our assumption that $P$ and $Q$ have no common factors. Therefore, $P'Q-PQ'$ and $Q^2$ have no common factors, and since $f'$ is nonvanishing, it is necessary that $P'Q-PQ'$ is a constant.

Take one more derivative of \eqref{claim}, we have 
\begin{align}
    \label{claim'}
P^{''}Q=PQ^{''}.    
\end{align}
Note that $f'$ nonvanishing also implies that $P$ has only simple roots. Let $z_0$ be any root of $Q$, then from \eqref{claim'} and since $P(z_0)$ cannot be zero, $Q^{''}(z_0)=0$. Hence the number of roots of $Q^{''}$ is less than or equal to the number of roots of $Q$. Since $Q$ is polynomial, it is necessary that $Q^{''}$ must be zero, and thus $Q$ is linear.

Similarly, $P^{''}=0$ and thus $P$ is linear. Therefore, $f$ is a Mobi\"us transform. In this case, by \eqref{complexrepresentation}, direct computation implies that $u$ is a radial solution up to translation.

If $f$ is not rational, then by Liouville's result, it is transcendental meromorphic. This finishes the proof.
\end{proof}
Now we are ready to prove Theorem \ref{lowerbounddiameterfor2D}.

\begin{proof}[Proof of Theorem \ref{lowerbounddiameterfor2D}]
Any solution $u$ to \eqref{scalliouville} can be written in the form
\begin{align}
    u=\ln \left(\frac{2|f'(z)|}{1+|f(z)|^2}\right),
\end{align}where $f$ is a meromorphic function such that $f'(z) \ne 0$ and $f$ has simple poles. As discussed before, if $f$ is rational, then $f$ must be a Mobi\"us transformation and hence $u$ is a bubble solution, and thus $g$ is the standard sphere metric. Hence $diam(\mathbb{R}^2)=\pi$ under such metric.

If $f$ is transcendental meromorphic, then by value distribution theory, $f$ takes the value in the complex plane infinitely many times except for two points. In particular, for any $\epsilon>0$ we can choose $z_1=(x_1,y_1)$ and $z_2=(x_2,y_2)$ such that $|f(z_1)|<\epsilon$ and $|f(z_2)|>1/\epsilon$. Let $\gamma(t)=z(t);\, a \le t \le b$ be a curve starting at $z_1$ and ending at $z_2$, and hence the length of $\gamma$ under the Euclidean metric is larger than $1/\epsilon-\epsilon$. Let $ \mathscr{H}^0(S)$ denote the
0-Hausdorff measure of a set $S\in \mathbb{R}^2$,  i.e. the number of points in $S$  if it is a finite set. 
Then the length of $\gamma$ under the conformal metric $e^{2u}\delta$ is given by
\begin{align*}
    \int_a^b e^{u(z(t))}|z'(t)|dt =&\int_a^b \frac{2|f'(z(t))|}{1+|f(z(t))|^2}|z'(t)|dt\\
    =& \int_{f(\gamma)}\frac{2}{1+|w|^2}\mathscr{H}^0\left(f^{-1}(w)\cap \gamma\right)|dw|\\
    \ge &\int_0^{1/\epsilon-\epsilon}\frac{2}{1+|w(s)|^2}ds, \quad \mbox{where $s$ is the arc length parameter}\\
    \ge & \int_{0}^{1/\epsilon-\epsilon}\frac{2}{1+(\epsilon+s)^2}ds, \quad \mbox{since $|w(s)|\le |w(0)|+|w(s)-w(0)|\le \epsilon+s$}\\
    =&2\left( \tan^{-1}(1/\epsilon)-\tan^{-1}(\epsilon)\right)\\
    \rightarrow & \pi \quad \mbox{as $\epsilon \rightarrow 0$.}
\end{align*}Hence $diam_g(\mathbb{R}^2) \ge \pi$ under the metric $g=e^{2u}\delta$.
\end{proof}

The proof above does not quite tell whether or not there exist two point $P,Q \in \mathbb{R}^2$ such that their conformal distance under the metric $g=e^{2u}\delta$ is bigger than or equal to $\pi$. Actually we have the following even stronger conclusion.

\begin{proposition}
\label{manypoints}
Let $u$ be a solution to \eqref{scalliouville} and $f(z)$ be its developing function. Then there exists a set $S$ such that $\mathscr{H}^0(\mathbb{C} \setminus S) \le 5$ and that for any point $A \in S$, we can find point $P \in f^{-1}(\{A\})$ and $Q \in \mathbb{C}$ such that $d_g(P,Q) \ge \pi$, where $d_g$ is the distance function under the conformal metric $g=e^{2u}\delta$.
\end{proposition}

\begin{proof}
Let $\Pi$ be the stereographic projection map from the north pole of unit sphere to the extended complex plane $\mathbb{C}\cup \{\infty\}$. Let $$X_1=\{A \in \mathbb{C}:f^{-1}\circ \Pi(\mbox{the antipodal of $\Pi^{-1}(A)) \ne \emptyset$} \},$$
$$X_2=\{A \in \mathbb{C}:\mbox{ the antipodal of $\Pi^{-1}(A)$ is not the north pole } \},$$
and
$$X=X_1\cap X_2.$$
Clearly $\mathscr{H}^0 (\mathbb{C} \setminus X) \le 2$ if $f$ is a Mobi\"us transform, and $\mathscr{H}^0 (\mathbb{C} \setminus X) \le 3$ if $f$ is transcendental meromorphic, by value distribution theory. 
Let $S=f(\mathbb{C}) \cap X$, and thus $\mathscr{H}^0(\mathbb{C} \setminus S) \le 5$.
For any $A \in S$, we can find $P \in f^{-1}(A)$. Moreover, by definition of $S$ we can find 
$$Q \in f^{-1}\circ \Pi(\mbox{the antipodal of $\Pi^{-1}(A)$}).$$
Let $\gamma$ be any curve in $\mathbb{C}$ from $P$ to $Q$, then the length of $\gamma$ is given by
\begin{align*}
    \int_{\gamma} \frac{2|f'(z)|}{1+|f(z)|^2} ds =& \int_{f(\gamma)}\frac{2}{1+|\omega|^2}\mathscr{H}^0\left(f^{-1}(\omega)\cap \gamma\right)|d\omega|\\
    \ge& \int_{f(\gamma)}\frac{2}{1+|\omega|^2}|d\omega|\\
    =&l_{S^2}\left(\Pi^{-1}(f(\gamma))\right)\\
    \ge & d_{S^2}(\Pi^{-1}\circ f(P),\Pi^{-1}\circ f(Q))\\
    =&d_{S^2}(\Pi^{-1}(A), \mbox{the antipodal of $\Pi^{-1}(A)$ })=\pi.
\end{align*}In the above $l_{S^2}$ is the length function on the unit sphere $S^2$ and $d_{S^2}$ is the sphere distance function. From the above estimate we immediately conclude that $d_g(P,Q) \ge \pi$. 
\end{proof}

\begin{remark}
If $u$ is a radial solution to \eqref{scalliouville}, then $diam_g(\mathbb{R}^2)=\pi$, since $(\mathbb{R}^2, e^{2u}\delta)$ is the standard unit sphere minus a point. 

One can also check that if $u$ is a $1$D solution to \eqref{scalliouville},  say $u(x,y)=\ln(\sech x)$, then under the corresponding conformal metric, $diam_g(\mathbb{R}^2)=\pi$, as we will see in the examples in next section. 
\end{remark}

We make one more remark on the $1D$ solutions:
\begin{remark}
\label{1dremark}
It is easy to prove that if $u$ is a $1$D solution, then its developing function $f(z)$ must have the form $f(z)=\frac{pe^{cz}-\bar{q}}{qe^{cz}+p}$, where $p,q\in \mathbb{C}$ and $|p|^2+|q|^2=1$. The converse is also true. 
\end{remark}

To the end of this section, we also remark that the diameter of $\mathbb{R}^2$ does not change under translation, rotation and scaling of solutions 
to the Liouville equation \eqref{scalliouville}, that is, if $g=e^{2u}\delta$ and 
$g_1=e^{2u_{\lambda, c,\omega}}\delta$ where $\lambda>0, c\in \mathbb{R}^2$, $\omega$ is a unit vector in $\mathbb{R}^2$ and
\begin{align*}
    u_{\lambda,c,\omega}=u(\lambda(\omega \cdot ((x,y)-c)))+\ln \lambda,
\end{align*}then $diam_{g}(\mathbb{R}^2)=diam_{g_1}(\mathbb{R}^2)$.

\section{Examples and Proof of Proposition \ref{utdiamter}}
In this section, we first show that for the family of solutions given by \eqref{u_t}, the diameters of $\mathbb{R}^2$ under the corresponding metrics can take all numbers in the interval $[\pi, 2\pi)$. Note that $u_t$ corresponds to the developing function $t+e^z$, and when $t=0$, $u_t$ is $1$D.

\begin{proof}[Proof of Proposition \ref{utdiamter}]
Since $u_t$ corresponds to the developing function $f(z)=t+e^z$, $u_t$ solves \eqref{scalliouville}.

To prove the diameter equality, first we note that 
\begin{align*}
    \int_{-\infty}^{\infty}e^{u_t(x,y)}dx=&\frac{2}{\sqrt{1+t^2\sin^2 y}}\tan^{-1}(\frac{t\cos y+e^x}{\sqrt{1+t^2 \sin^2 y}})\Big|_{-\infty}^{\infty}\\
    =&\frac{2}{\sqrt{1+t^2\sin^2 y}}\left(\frac{\pi}{2}-\tan^{-1}(\frac{t\cos y}{\sqrt{1+t^2 \sin^2 y}})\right).
\end{align*}
Hence \begin{align}
    \label{sup'}
\sup_{y \in \mathbb{R}}\int_{-\infty}^{\infty}e^{u_t(x,y)}dx=\pi+2\tan^{-1}(t).    
\end{align}
Given arbitrary two points $P_1, P_2 \in \mathbb{R}^2$, for each point $P_i\, (i=1,2)$, we let $\gamma_R^i$ be the horizontal line segment passing through $P_i$ with $Q_{-R}^i$ and $Q_R^i$ as the left and right end points of $\gamma_R^i$, such that $Q_{-R}^1 Q_R^1 Q_R^2Q_{-R}^2$ is a rectangle with length $2R$. Let $\Gamma_{R}$ be the vertical line segments connecting $Q_R^1$ and $Q_R^2$, and $\Gamma_{-R}$ be the vertical line segment connecting $Q_{-R}^1$ and $Q_{-R}^2$. By \eqref{sup'}, $l(\gamma_{R}^i) \le \pi+2\tan^{-1}(t)$ for $i=1,2$. Also, it is easy to see that $\lim_{R \rightarrow \pm\infty} l(\Gamma_{\pm R})=0$. Now that
\begin{align*}
    2d_g(P_1,P_2) \le  \sum_{i=1}^2l (\gamma_{R}^i)+ l(\Gamma_R)+l(\Gamma_{-R}),
\end{align*}by letting $R \rightarrow \infty$, we have that $d_g(P_1,P_2) \le \pi+2\tan^{-1}(t)$. Since $P_1$ and $P_2$ are arbitary, we have that $diam_g(\mathbb{R}^2) \le \pi+2\tan^{-1}(t)$.

To show that the equality holds, we choose $P_1=(a,\pi)$ and $P_2=(a,-\pi)$, where $a$ satisfies
\begin{align}
\label{newa}
    e^a-t=\tan(\frac{\pi}{4}-\frac{1}{2}\tan^{-1}(t)).
\end{align}
Such $a$ exists since $t+\tan(\frac{\pi}{4}-\frac{1}{2}\tan^{-1}(t)) \in [1,\infty)$ if $t \in [0,\infty)$. Let $\Pi$ be the stereographic projection map from the north pole of the unit sphere. Since $f(P_1)=f(P_2)=(t-e^a,0)$, by \eqref{newa} and double angle formula, for $i=1,2$ we have
\begin{align*}
    \Pi^{-1}(f(P_i))=(\frac{2(t-e^a)}{1+(t-e^a)^2},0,\frac{(t-e^a)^2-1}{(t-e^a)^2+1})=(-\sin(\frac{\pi}{2}-\tan^{-1}(t)),0,-\cos(\frac{\pi}{2}-\tan^{-1}(t))).
\end{align*}Let $\alpha=\tan^{-1}(t)$, hence $\alpha \in [0,\frac{\pi}{2})$ and
\begin{align}
\label{P_1}
    \Pi^{-1}(f(P_1))=\Pi^{-1}(f(P_2))=(-\cos \alpha,0,-\sin\alpha).
\end{align}
Let $\gamma$ be a curve connecting $P_1$ and $P_2$, then $\gamma$ must intersect the $x$-axis. Let us assume that $\gamma$ passes through $P_3=(b,0)$, and thus $f(b)=t+e^b>t$. Let $f(b)=\tan \beta$ for some $\beta \in (0,\frac{\pi}{2})$. So $\beta >\alpha$. and 
\begin{align}
\label{P_3}
    \Pi^{-1}(f(P_3))=(\frac{2\tan\beta}{1+\tan^2\beta},0,\frac{\tan^2\beta-1}{1+\tan^2\beta})=(\sin(2\beta),0,-\cos(2\beta)).
\end{align}
Let $\theta \in (0,\pi)$ be the angle between $\Pi^{-1}(f(P_1))$ and $\Pi^{-1}(f(P_3))$. Then by \eqref{P_1} and \eqref{P_3}, 
\begin{align}
\label{costheta}
    \cos \theta=-\cos \alpha \sin(2\beta)+\sin \alpha \cos(2\beta)=\cos(\frac{\pi}{2}+2\beta-\alpha).
\end{align}
If $\pi/2+2\beta-\alpha<\pi$, then by \eqref{costheta} we know $\theta=\pi/2+2\beta-\alpha>\pi/2+\alpha$.
If $\pi/2+2\beta-\alpha>\pi$, then since $\beta <\frac{\pi}{2}$, we still have
\begin{align*}
    \theta=2\pi- (\pi/2+2\beta-\alpha)=\frac{3\pi}{2}+\alpha-2\beta>\frac{\pi}{2}+\alpha.
\end{align*}
Therefore,
\begin{align}
    \label{chonglai}
 d_{g_{S^2}}(\Pi^{-1}(f(P_1)),\Pi^{-1}(f(P_3)))\ge \frac{\pi}{2}+\alpha.    
\end{align}
By \eqref{complexrepresentation}, we have 
\begin{align}
\label{huanyuan}
    l(\gamma)=\int_{f(\gamma)}\frac{2}{1+|\omega|^2}\mathscr{H}^{0}(f^{-1}(\omega)\cap \gamma)|d\omega|
    \ge\int_{f(\gamma)}\frac{2}{1+|\omega|^2}|d\omega|.
\end{align}
Note that the metric of the unit sphere is given by $g_{S^2}=\frac{4}{(1+|\omega|^2)^2}\delta$, where $\omega$ denotes the coordinate of point on sphere obtained using stereographic projection map from the north pole. Hence from \eqref{huanyuan}, \begin{align}
    \label{chonglai2}
l(\gamma)\ge l_{g_{S^2}}(\Pi^{-1}(f(\gamma)),    
\end{align}
where $l_{g_{S^2}}$ is the length function on the unit sphere. 

Since  $\Pi^{-1}\circ f (\gamma)$ is a curve in $S^2$ such that it starts from $\Pi^{-1}(f(P_1))$, goes through $\Pi^{-1}(f(P_3))$ and then goes back to $\Pi^{-1}(f(P_2))=\Pi^{-1}(f(P_1))$, by \eqref{chonglai} and \eqref{chonglai2} we know that the length of $\gamma$ satisfies
\begin{align*}
    l(\gamma) \ge 2d_{g_{S^2}}(\Pi^{-1}(f(P_1)),\Pi^{-1}(f(P_3)))\ge \pi+2\alpha.
\end{align*}
That is, $l(\gamma)\ge \pi+2\tan^{-1}(t)$. Hence $diam_g(\mathbb{R}^2) \ge  \pi+2\tan^{-1}(t)$.

Therefore, we have shown that $diam_g(\mathbb{R}^2) = \pi+2\tan^{-1}(t)$.
\end{proof}

Here we make a remark. At the beginning of this project, we only knew that there exist a family of solutions given by \eqref{u_t}, and by just naively looking at the horizontal integrals, we figured out the upper bound for the corresponding diameters. It is quite straightforward up to this step. However, it took us quite a while to find out a way of exactly computing the conformal diameters corresponding to these functions $u_t$. Let us use the following example to geometrically illustrate the case $t=1$. In fact, this example is so important to us, that only after understanding it did we observe Proposition \ref{manypoints} and have a better understanding of the conformal diameters corresponding to solutions to Liouville equation \eqref{scalliouville}. So we will present the full details including some numerical computations as motivations in the example.

\begin{example}
\label{bifurcationof1dsolution}
Let \begin{align}
\label{1+e^z}
    u(x,y)=u_1(x,y)=\ln\left(\frac{2e^x}{2+2e^x\cos y+e^{2x}}\right).
\end{align}Then clearly $u$ solves \eqref{scalliouville}.  In the following, we will gradually show that $diam_g(\mathbb{R}^2)=\frac{3\pi}{2}$ if $g=e^{2u}\delta$ for such $u$. 

First, note that 
\begin{align}
\label{jisuan1}
    \int_{-\infty}^{\infty} e^{u(x,y)}dx=\frac{2}{\sqrt{1+\sin^2 y}}\tan^{-1}\left(\frac{\cos y+e^x}{\sqrt{1+\sin^2 y}}\right)+C
\end{align}and that
\begin{align}
\label{jisuan2}
    \int_0^{\pi} e^{u(x,y)}dy =\frac{\pi}{\sqrt{(e^{-x}+\frac{e^x}{2})^2-1}}.
\end{align}
By \eqref{jisuan1}, we have that
\begin{align}
\label{sup}
    \sup_{y \in \mathbb{R}} \int_{-\infty}^{\infty} e^{u(x,y)} dx=\frac{3\pi}{2},
\end{align}
and hence by exact argument in the proof of Proposition \ref{utdiamter}, we have that $diam_g(\mathbb{R}^2) \le \frac{3\pi}{2}$. So the question is, can the equality be attained?

Let us consider somehow the worst case. By \eqref{jisuan1}, the supremum of \eqref{sup} is attained at $y=(2k+1)\pi, \, k \in \mathbb{Z}$. Let us choose two points $P_1=(a,\pi)$ and $P_2=(a,-\pi)$, where $a \in \mathbb{R}$ such that 
\begin{align}
\label{intchoiceofa}
    \int_{-\infty}^a e^{u(x,\pm\pi)}dx=\int_a^{\infty} e^{u(x,\pm\pi)}dx=\frac{3\pi}{4}.
\end{align}One can see that $P_1$ is chosen to lie in the ``middle" way of the horizontal line from $(-\infty,\pi)$ to $(\infty,\pi)$, and similarly for $P_2$. From the choices, we can expect that the distance between $P_1=(a,\pi)$ and $P_2=(a,-\pi)$ is equal to $\frac{3\pi}{2}$ under the metric $g=e^{2u}\delta$. This guess can be supported by computing the length of ellipses connecting the two points:

Let $C_s\, (s>0)$ be the right half of the ellipses to connect $(a,\pi)$ and $(a,-\pi)$, and thus the equation of $C_s$ is given by:
\begin{align*}
    \frac{(x-a)^2}{s^2}+y^2=\pi^2.
\end{align*}
By setting $x=a+\pi \cos \theta$ and $y=\pi \sin \theta$, the length of $C_s$ under metric $g$ is given by
\begin{align}
\label{lcs}
    l(C_s)=\int_{-\pi/2}^{\pi/2}\frac{2\pi e^{a+s\pi \cos \theta}}{2+2e^{a+s\pi \cos \theta}\cos(\pi \sin \theta)+e^{2a+2s\pi \cos \theta}} \sqrt{s^2\sin^2\theta+\cos^2\theta}\, d\theta.
\end{align}
By \eqref{jisuan1} and \eqref{intchoiceofa}, we have
\begin{align}
\label{defofa}
    tan^{-1}(e^a-1)=\frac{\pi}{8}.
\end{align}
The following is the graph obtained by Mathematica of the function $l(C_s)-\frac{3\pi}{2}$ for $s \in [0,10]$.\\

\includegraphics[scale=1]{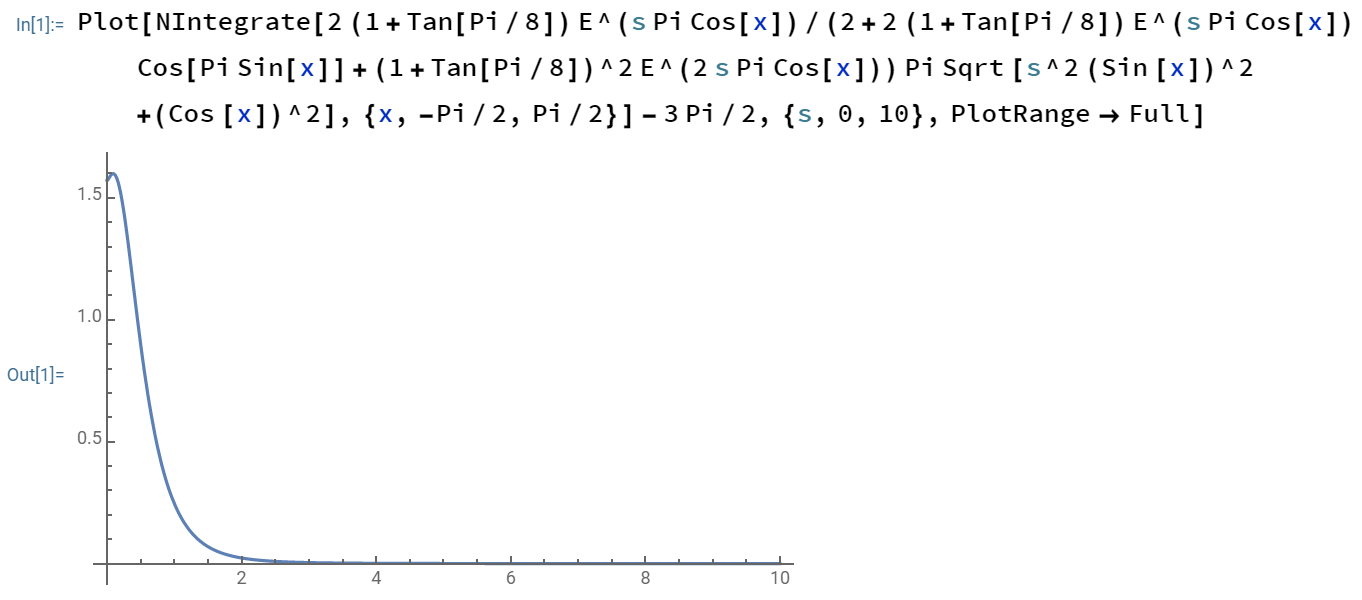}\\

The graph indicates that $l(C_s) \ge \frac{3\pi}{2}, \, \forall s>0$, and $\lim_{s \rightarrow \infty}l(C_s)=\frac{3\pi}{2}$. This suggests that $d_g(P_1, P_2) = \frac{3\pi}{2}$.  To rigorously prove it,  we proceed with the exact same proof as in that of Proposition \ref{utdiamter}, for $t=1$.
\begin{proof}
Let $\gamma$ be a curve starting at $P_1$ and ending at $P_2$, then $\gamma$ must pass through some point on $x$-axis, which is given by $P_3=(b,0)$. Let $P_i'=f(P_i)$, where $f(z)=1+e^z$, which is the developing function for the solution $u$. By \eqref{defofa}, $P_1'=P_2'=(-\tan(\pi/8),0)$. Also, $P_3'=(b',0)$, where $b'=1+e^b>1$. Then as in the proof of Proposition \ref{utdiamter}, we have $l(\gamma)\ge l_{g_{S^2}}(\Pi^{-1}(f(\gamma))$. 

By direct computation, $\Pi^{-1}(P_1')=\Pi^{-1}(P_2')=(-\frac{\sqrt{2}}{2},0,-\frac{\sqrt{{2}}}{2})$, and we denote the point by $A$. Hence $\Pi^{-1}(f(\gamma))$ must be a closed curve on the unit sphere such that it starts and ends at the point $A$, and it must passes through the point $B=\Pi^{-1}(P_3')$, which lies in the circle connecting $(1,0,0)$ and the north pole in the northern hemisphere. Therefore, $\Pi^{-1}(\gamma) \ge 2dist_{g_{S^2}}(A,B) \ge \frac{3\pi}{2}$. Hence $diam_g(\mathbb{R}^2)=\frac{3\pi}{2}$.
\end{proof}

An interesting fact is that, from the proof above, we can now rigorously prove that $$\lim_{s \rightarrow \infty}l(C_s)=\frac{3\pi}{2}.$$ Indeed, when $s$ goes to $+\infty$, $\Pi^{-1}(f(C_s))$ is getting closer and closer to the curve that starts from $(-\frac{\sqrt{2}}{2},0,-\frac{\sqrt{{2}}}{2})$, goes along the great circle to the north pole, and then goes back to $(-\frac{\sqrt{2}}{2},0,-\frac{\sqrt{{2}}}{2})$. Hence the total length is closer and closer to $\frac{3\pi}{2}$. 

Without using the geometry on unit sphere, it seems very difficult to handle the ``monster" integral \eqref{lcs}. 

This finishes the  discussion of  Example 3.1.
\end{example}

Next, we construct a solution to \eqref{scalliouville} such that the corresponding conformal diameter can be greater than or equal to $2\pi$.
\begin{example}
\label{2pi}
Let \begin{align}
    \label{infinitydiameter}
u(x,y)=\ln \left(\frac{2e^{x+e^x \cos y}}{1+e^{2e^x \cos y}}\right).   
\end{align}We will show that $u$ solves \eqref{scalliouville}, and that $diam_g(\mathbb{R}^2)\ge  2\pi$, where $g=e^{2u}\delta$.

\begin{proof}
First, such $u$ given by \eqref{infinitydiameter} corresponds to the developing function 
\begin{align*}
    f(z)=e^{e^z}.
\end{align*}
Hence $u$ is a solution to \eqref{scalliouville}.

Next, note that \begin{align*}
    f(z)=e^{e^x\cos y}\left(\cos (e^x \sin y)+i \sin (e^x \sin y)\right).
\end{align*}We consider the distance between the point $P=(\ln \pi,\frac{\pi}{2})$ and point $Q=(\ln \pi, -\frac{3\pi}{2})$. We have $f(P)=f(Q)=-1$ and $P':=\Pi^{-1}(f(P))=\Pi^{-1}(f(Q))=(-1,0,0)$, where $\Pi$ is the stereographic projection map from the north pole. 
Let $\gamma$ be a curve starting from $P$ to $Q$, then $\gamma$ must pass through a point $P_1=(b,0)$ and another point $P_2=(c,-\pi)$. Since $f(P_1)=e^{e^b}>1$, $P_1':=\Pi^{-1}\circ f(P_1)$ lies between the arc from $(1,0,0)$ to the north pole. Since $f(P_2)=e^{-e^c} \in (0,1)$, $P_2':=\Pi^{-1}\circ f(P_2)$ lies between the arc from $(1,0,0)$ to the south pole.

Let $d_g$ be the distance function under the metric $g=e^{2u}\delta$, $d_{S^2}$ be the distance function on the unit sphere $S^2$, and $l_{S^2}$ be the length function on the unit sphere. By the proof of Proposition \ref{manypoints} or Proposition \ref{utdiamter}, and since $\Pi^{-1}\circ f(\gamma)$ is a curve on the unit sphere starting from $P'$, passing through $P_1', P_2'$ and ending at $P'$, we have
\begin{align*}
    d_g(P,Q) \ge l_{S^2}(\Pi^{-1} \circ f (\gamma)) \ge d_{S^2}(P', P_1')+d_{S^2}(P_1',P_2')+d_{S^2}(P_2',P').
\end{align*}By the location of $P', P_1'$ and $P_2'$, we have
\begin{align*}
     d_g(P,Q) \ge 2\pi.
\end{align*}Hence $diam_g(\mathbb{R}^2 ) \ge 2\pi$.

\end{proof}

\end{example}

\section{Diameter upper bound for supersolutions to \eqref{scalliouville}}
In this section, our main goal is prove Theorem \ref{weakbounddiamter}.

In the following, we will constantly use $d_g(x,y)$ to denote the distance between two points $x$ and $y$ under the metric $g$, that is, the infimum of the lengths of piecewise smooth curves starting at $x$ and ending at $y$. Unlike previous sections, in this section $x$ and $y$ are denoted as points in Euclidean spaces, not coordinate components.

First, for the noncompact Riemannian manifold $(\mathbb{R}^n, g)$, it is convenient to introduce the following definition.
\begin{definition}
For any $p \in \mathbb{R}^n$, $d_g(p,\infty)$ is defined as 
\begin{align*}
    \inf\{l(\gamma): \mbox{$\gamma$ is a piecewise smooth curve from [0,1) to $\mathbb{R}^n$ with $\gamma(0)=p$ and $|\gamma(1^-)|=\infty$}\},
\end{align*}where $l$ is the length function under metric $g$ and $|\cdot|$ is the Euclidean norm. 
\end{definition}

Before stating certain geometric results related to $(\mathbb{R}^2, e^{2u}\delta)$, we first prove the following key observation, which does not rely on finiteness assumption on conformal volume. Actually the statement can be made in $\mathbb{R}^n$ as follows:
\begin{proposition}
\label{diametercontrolinfty}
Let $g$ be a Riemannian metric on $\mathbb{R}^n$ with $Ric_g \ge (n-1)g$, then we have $d_g(x, \infty)\le \pi$ for all $x \in \mathbb{R}^n$. 
\end{proposition}

Proposition \ref{diametercontrolinfty} is related to Myer's Theorem. Recall that Myer's Theorem states that if $(M,g)$ is a geodesically complete manifold with $Ric \ge (n-1)$, then $diam(M) \le \pi$. Moreover, if the diameter is equal to $\pi$, then by Cheng's rigidity result in \cite{Cheng75} , $M$ must be isometric to sphere. 

However, the proof of Myer's theorem requires that for any $x, y \in M$, there is a minimizing geodesic connecting $x$ and $y$. For incomplete manifolds, generally two points cannot be connected by a length minimizing curve. So in order to prove Proposition \ref{diametercontrolinfty}, we need to somehow overcome the non-completeness issue. The proof is motivated by the idea in the proof of Hopf-Rinow theorem.

\begin{proof}[Proof of Proposition \ref{diametercontrolinfty}]
Fix $x \in \mathbb{R}^n$ and $r>0$, and $B_r(x)$ be the Euclidean $r$-neighborhood of $x$. Let $y \in \partial B_r(x)$ such that $d_g(x,y)=d_g(x,\partial B_r(x))=T$. Let $\gamma:[0,b]$ be the unit speed curve starting at $x$. We say $\gamma|_{[0,b]}$ \textit{aims at} $y$ if $\gamma$ starts at $x$ and that for any $0<t<b$, $l(\gamma|_{[0,t]})+d_g(\gamma(t),y)=d_g(x,y)$. Clearly $b \le T$. We claim that there exists such a curve starting at $x$ and aiming at $y$ with length equal to $T$.

We prove this claim by modifying the proof of Hopf-Rinow Theorem. Let $$S=\{b \in [0,T]: \mbox{there exists a unit speed curve $\gamma$ starting at $x$ such that $\gamma|_{[0,b]}$ aims at $y$}\}$$ and let $T_0=\sup S$. 

First, note that $S$ is not empty. This is because we can always choose a small geodesic $\epsilon$-neighborhood of $x$, denoted $B'_{\epsilon}(x)$. Since $d_g(\cdot,y)$ is continuous, there is $z \in \partial B'_{\epsilon}(x)$ such that $d_g(z, y)= d_g(y,\partial B'_{\epsilon})$. Let $\gamma$ be the minimizing geodesic connecting $x$ and $z$, and thus $\gamma|_{[0,\epsilon]}$ aims at $y$.

If $T_0<T$, then we claim that $\gamma|_{[0,T_0]}$ lies in a compact subset of $B_r(x)$. Indeed, if this is not the case, then there exists $w\in \partial B_r(x)$ such that $\gamma(t)=w$ for some $t \le T_0$. Then $d_g(x,w)\le T_0<T=d_g(x,y)$, which is a contradiction for the choice of $y$ above. Therefore, $\gamma|_{[0,T_0]} \Subset B_r(x)$. Hence again we can choose a small geodesic neighborhood of $\gamma(T_0)$ and this extends the length of the curve starting at $x$ aiming at $y$. Therefore, eventually we conclude that $T_0=T$.

Note that such $\gamma$ is actually a length minimizing geodesic connecting $x$ and $y$. Therefore, by the proof of Myer's Theorem, see for example \cite{Lee}, we have $l(\gamma) \le \pi$. Hence $d_g(x, \partial B_r(x)) \le \pi$. Since $r$ arbitrary, let $r \rightarrow\infty$ we have $d_g(x,\infty) \le \pi$. 
\end{proof}

\begin{remark}
\label{keyobservation}
The essential observation in the proof of of Proposition \ref{diametercontrolinfty} is that for any $x \in \mathbb{R}^n$ and $r>0$, we can find a point $y \in \partial B_r(x)$ such that there is a length minimizing curve connecting $x$ and $y$. We will use this fact often times in this paper. Note that this is generally not true for every $y \in \partial B_r(x)$.
\end{remark}

\begin{remark}
\label{chengshi}
Also from the proof above we know for any $x \in \mathbb{R}^n$, there exists a unit speed curve $\gamma$ starting at $x$ and aiming at $\infty$, that is, for any $0<t<d_g(x,\infty)$, $l(\gamma_{[0,t]})+d_g(\gamma(t),\infty)=d_g(x,\infty)$.
\end{remark}

\begin{remark}
The other triangle inequality $d_g(x,\infty)+d_g(y,\infty)\ge d_g(x,y)$ generally does not hold. This can be seen in $(\mathbb{R}^2,e^{2u}\delta)$, where $u(t)=\ln(\sech t)$ is the $1D$ solution to \eqref{scalliouville}. Let $x=(-R,0)$ and $y=(R,0)$, then $d_g(x,y)\rightarrow \pi$ as $R \rightarrow \infty$ while both $d_g(x,\infty)$ and $d_g(y,\infty)$ converge to $0$.
\end{remark}

Applying Proposition \ref{diametercontrolinfty} and assuming the finiteness of volume of $(\mathbb{R}^2,g)$, we now prove Theorem \ref{weakbounddiamter}.

\begin{proof}[Proof of Theorem \ref{weakbounddiamter}]
We first show there exists a sequence $r_k \rightarrow \infty$ such that $\int_{\partial B_{r_k}} e^{u} ds \rightarrow 0$. Suppose this is not the case, then there is $c>0$ such that for any $r>0$, $\int_{\partial B_{r}} e^{u} ds>c$. Hence by H\"older's inequality,
\begin{align*}
    2\pi r\int_{\partial B_r} e^{2u} ds \ge \left(\int_{\partial B_{r}} e^{u} ds\right)^2 \ge c^2.
\end{align*}
Hence $\int_{\partial B_r}e^{2u} ds \ge \frac{c^2}{2\pi r}$, and this would imply $\int_{\mathbb{R}^2} e^{2u} dx=\infty$. Thus we get a contradiction.

For any $x,y \in \mathbb{R}^2$, by Proposition \ref{diametercontrolinfty} and Remark \ref{chengshi}, we can choose curves $\gamma_1$ and $\gamma_2$ starting at $x$ and $y$ respectively aiming at $\infty$ with lengths less than or equal to $\pi$. Let $z_{k_i}=\partial B_{r_k} \cap \gamma_i,\, i=1,2$, where $r_k$ is the sequence above. Hence we have \begin{align*}
    dist_g(x,y)\le dist(x,z_{k_1})+dist(y,z_{k_2}) +dist_g(z_{k_1},z_{k_2})\le 2\pi+\int_{\partial B_{r_k}} e^{u} \rightarrow 2\pi
\end{align*}as $k \rightarrow \infty$. Therefore, $diam(\mathbb{R}^2)\le 2\pi$. 
\end{proof}

\section{geometric inequalities for radial supersolutions to \eqref{scalliouville}}
We first prove the following ordinary differential inequality results.
\begin{proposition}
\label{aiya}
Let $u=u(r)$ be a function satisfying 
\begin{align}
    \label{2dradialcase}
u^{''}+\frac{u'}{r}+e^{2u}\le 0,    
\end{align} then
\begin{align}
    \label{radlength}
 \int_0^{\infty} e^{u(r)}dr \le \pi.    
\end{align}
\end{proposition}
Unfortunately so far we cannot give a analytic proof of this result, so we use geometric argument to prove this proposition.
\begin{proof}
At any $x \in \mathbb{R}^2$, we consider $g_x=e^{2u(|x|)}\delta_x$,where $u$ is a solution to \eqref{2dradialcase}. This gives a metric $g$ in $\mathbb{R}^2$ with $Ric_g \ge g$. 

First note that all the rays starting from the origin are geodesics. This is because, given any curve $\gamma$, the geodesic curvature of $\gamma$ is given by $\kappa_g=e^{-u}(\kappa+\partial u/\partial \nu)$, where $\kappa$ is the curvature of $\gamma$ in the Euclidean metric, and $\nu$ is the unit normal to $\gamma$. So if $\gamma$ is a ray, then since $u$ is radial, $\partial u/\partial \nu$ is zero, and of course $\kappa$ is also zero, so $k_g$ is zero. Hence rays starting from origin are geodesics.

By local uniqueness of geodesics, all geodesics  starting from the origin must be the rays. 

Next, we claim that for any $x \in \mathbb{R}^2$, there is a minimizing geodesic connecting $0$ and $x$. Indeed, by Remark \ref{keyobservation}, there is a minimizing geodesic connecting $0$ and $\partial B_{|x|}$. Since $u(|\cdot|)$ is a radial function, by symmetry there is a minimizing geodesic connecting $0$ and $x$. 

Since all the geodesics from $0$ are rays and no two rays intersect at other points except $0$, such minimizing geodesic must be part of the rays starting from $0$ and passing through $x$. Therefore, by the proof of Myer's Theorem, $\int_0^{|x|}e^{u(r)}dr \le \pi$. Sending $|x| \rightarrow \infty$, we have \eqref{radlength}.
\end{proof}

\hfill\\

Recall that $l(r):=\int_{\partial B_r} e^{u} ds$, which is the conformal length of $\partial B_r$, and that $A(r):=\int_{B_r} e^{2u} dx$, which is the conformal area of $B_r$. Now we prove Theorem \ref{maintheoremonradialcase}.

\begin{proof}[Proof of Theorem \ref{maintheoremonradialcase}]
We first prove the second inequality in \eqref{leibi}:
Since $u$ is radial, we may write $u(x)=u(r)$, where $r=|x|$. By \eqref{ricci1} and integration by part, we have
\begin{align}
\label{yanxing1}
    A \le -u' 2\pi r.
\end{align}
Since $A'=2\pi re^{2u}$, by \eqref{yanxing1} we have
\begin{align*}
    2A A'\le -2u'e^{2u}(2\pi r)^2 =&-(e^{2u})'4\pi^2 r^2\\
    =& -(e^{2u}4\pi^2 r^2)'+e^{2u}8\pi^2 r\\
    =& -(e^{2u}4\pi^2r^2)'+4\pi A'.
\end{align*}
Hence \begin{align*}
    \left(4\pi A -A^2\right)'\ge \left(e^{2u}4\pi^2 r^2\right)'.
\end{align*}Integrating the above from $0$ to $r$, we have
\begin{align}
\label{yanxing2}
    4\pi A(r) -A^2(r) \ge \left(2\pi r e^{u(r)}\right)^2=l^2(r).
\end{align}This proves the second inequality of \eqref{leibi}.
\hfill\\

Next, we prove \eqref{volumeradial}. This is actually from the above inquality, since it immediately implies that $A(r) \le 4\pi$ for any $r>0$. Let $r\rightarrow \infty$, we conclude that $vol_g(\mathbb{R}^2) =\int_{\mathbb{R}^2} e^{2u} \le 4\pi$.
\hfill\\

Next, we prove the results related to properties of $l(r)$:
First, \eqref{yanxing3} also follows from \eqref{yanxing2}, since $4\pi A-A^2 \le \left(\frac{4\pi-A+A}{2}\right)^2 =4\pi^2$. Hence $l(r) \le 2\pi$ for any $r>0$.

Note that $l'(r)=2\pi e^{u(r)}(1+ru'(r))$ and that $(1+ru'(r))'\le -re^{2u}<0$. If $1+ru'(r)\ge 0$ for all $r>0$, then $l'(r) \ge 0$ and hence $l(r)$ is increasing. This cannot happen since by the proof of Theorem \ref{weakbounddiamter} we can choose a sequence $r_k$ such that $l(r_k) \rightarrow 0$. Hence there exists $r_0>0$ such that when $r\le r_0$, $1+ru'(r) \ge 0$ and when $r>r_0$, $1+ru'(r)\le 0$. Hence $l(r)$ is increasing when $r<r_0$, reaches its maximum at $r_0$ and then decreasing when $r>r_0$. 

Since $l(r_k) \rightarrow 0$ as $k \rightarrow \infty$ and $l(r)$ is decreasing when $r>r_0$, \eqref{jiandanyouyong} is also proved.
\hfill\\

Next, we prove \eqref{boundsforr0} by applying the Heintze-Karcher inequality, which is originally obtained in \cite{HK78}, see also \cite{GLH}[Theorem 4.21] : For any $0<r \le r_0$, we have
\begin{align}
\label{zhuanzhu1}
    l(r) \le \int_{\partial B_{r_0}}\Big(\cos (R(r_0)-R(r))-\eta(r_0) \sin (R(r_0)-R(r)) \Big)e^{u(r_0)} ds,
\end{align}where the function $\eta(\rho)$ is the mean curvature of $\partial B_{\rho}$ under the metric $g=e^{2u}\delta$. Note that even though the statement of Heintze-Karhcer inequality is for domains in complete manifolds, the proof only requires that any point inside the domain can be connected to the boundary along exponential map. This is true in our case since as shown in the proof Proposition \ref{aiya}, any line segment belonging to the ray starting from the origin must be a minimizing geodesic. 

From \eqref{zhuanzhu1} we have that for any $0<r\le r_0$,
\begin{align}
    \cos (R(r_0)-R(r))-\eta(r_0) \sin (R(r_0)-R(r)) \ge 0.
\end{align}
Since $\eta=e^{-u}(u'+\frac{1}{r})$ and $l'(r_0)=0$, $\eta(r_0)=0$. Hence $\cos (R(r_0)-R(r)) \ge 0$ for any $0<r \le r_0$. Therefore, $R(r_0) \le \frac{\pi}{2}$. This is \eqref{boundsforr0}. 
\hfill\\

Now let us prove the first inequality of \eqref{leibi}, which is also a consequence of Heintze-Karcher inequality. We prove as follows. For any $r>0$, applying Heintze-Karcher inequality on $(B_r, e^{2u}\delta)$ and integrating from $0$ to $r$, we have
\begin{align}
\label{zhujidan}
    A(r) \le l(r) \int_0^R (\cos t-\eta(r) \sin t)dt,
\end{align}where $R=R(r)=\int_0^r e^{u(\rho)}d\rho$. Hence $\eta(r) \le \cot r$ for $r \le R$. Let $\eta(r)=\cot \xi$, where $\xi\in [0,\pi]$, and thus $\xi \ge R$ and when $0\le t \le \xi$, $\cos t-\eta(r) \sin t \ge 0$. Hence from \eqref{zhujidan} we have
\begin{align}
\label{zhujidan'}
    A(r) \le l(r) \int_0^{\xi} (\cos t-\eta(r) \sin t)dt.
\end{align}
Since all the radial rays starting $\partial B_r$ exhausts $\mathbb{R}^2 \setminus B_r$, we can also apply Heintze-Karcher inequality on $\mathbb{R}^2 \setminus B_r$ to get
\begin{align}
\label{zhujidan1}
    A_{\infty}-A(r) \le l(r) \int_0^{R_{\infty}-R}(\cos t+\eta \sin t)dt,
\end{align}where $R_{\infty}=\int_0^{\infty} e^{u(\rho)}d\rho$. By Proposition \ref{aiya}, $R_{\infty} \le \pi$.
Suppose that $t \in [0,\pi]$, then $\cos t+\eta(r) \sin t =\sin t (\cot t+\cot \xi) \ge 0$ if and only if $0 \le t \le \pi-\xi$. Note also that from \eqref{zhujidan1} we have that 
\begin{align}
    \label{faxian}
\cos t-\eta(r) \sin t \ge 0, \quad \forall 0\le t \le R_{\infty}-R.    
\end{align}
Then since $\eta(r)=\cot \xi$ and the integrand is nonnegative, we have 
\begin{align}
    \label{Rinftyestimate}
R_{\infty}-R \le \pi-\xi. 
\end{align}Therefore, we have
\begin{align}
\label{zhujidan1'}
    A_{\infty}-A(r) \le l(r) \int_0^{\pi-\xi}(\cos t+\eta \sin t)dt.
\end{align}
Multiplying \eqref{zhujidan'} and \eqref{zhujidan1'}, we have
\begin{align*}
    A(A_{\infty}-A) \le & l^2 \left( \int_0^{\xi} (\cos t-\eta \sin t)dt\right) \left(\int_0^{\pi-\xi} (\cos t+\eta \sin t) dt \right)\\
    =& l^2 \left(\sin \xi +\cot \xi (\cos \xi -1)\right)\left(\sin \xi +\cot \xi (\cos \xi +1)\right)\\
    =& l^2 \frac{1-\cos \xi}{\sin \xi } \frac{1+\cos \xi}{\sin \xi}=l^2.
\end{align*}This proves the first inequality of \eqref{leibi}. 
\hfill\\

Next, we prove \eqref{hmm}. It is actually a consequence of Bishop-Gromov inequality: Let $B_R'$ denote the geodesic ball of radius $R$ centered at the origin, where $R=R(r)=\int_0^r e^{u(\rho)}dr$, then 
\begin{align*}
    f(r):=\frac{vol_g(B_R')}{vol_{S^2}(B_R')}=\frac{\int_{B_r}e^{2u}dx}{2\pi (1-\cos R(r))}
\end{align*}is a non-increasing function. Note that even if $(\mathbb{R}^n, e^{2u}\delta)$ is not complete, we can still apply this inequality because the proof only requires that for any point on $\partial B_R$, there is a minimizing geodesic connecting the origin and the point. This is true since $u$ is radial. We have also used the fact that the line segment starting from the origin to any point on $\partial B_r$ is a minimizing geodesic, as shown in the proof of proposition \ref{aiya}.

Now that $f(r)$ is non-increasing, $f'(r)\le 0$. By directly computing $f'(r)$, we have
\begin{align}
\label{jinzhipaixu}
    2\pi re^{2u(r)}2\pi(1-\cos R) -\left(\int_{B_r}e^{2u}dx\right)2\pi (\sin R) e^{u(r)} \le 0
\end{align}
Since $A(r)=\int_{B_r}e^{2u}dx$ and $l(r)=2\pi r e^{u(r)}$, \eqref{jinzhipaixu} therefore implies \eqref{hmm}.
\hfill\\

Next, we prove \eqref{jiefaqian} and \eqref{jiefa}. Recall that by the Heintze-Karcher inequality, we have
\begin{align}
\label{radialHeintzKarcher}
    A(r) \le \int_0^R \int_{\partial B_r} (\cos t-\eta \sin t)ds dt,
\end{align}where $R=R(r)=\int_0^r e^{u(\rho)}d\rho$ and $\eta$ is the geodesic curvature of $\partial B_r$, which is equal to $e^{-u}(\frac{1}{r}+u'(r))$. Hence by \eqref{radialHeintzKarcher} we have \begin{align}
    A \le & l \sin R-\eta l (1-\cos R)\\
    =&l \sin R -e^{-u}(1/r+u'(r))2\pi r e^{u(r)}(1-\cos R)\\
    =\label{withoutusingbishop} & l\sin R-e^{-u} l' (1-\cos R).
\end{align}
Since $l' \ge 0$ when $r \le r_0$, hence $A \le l \sin R$, and this proves \eqref{jiefaqian}. \eqref{jiefa} is a direct consequence of \eqref{jiefaqian}.
\hfill\\

Next, we prove \eqref{busanhuang} and \eqref{bushanhuangtuichu}. By \eqref{hmm} and the first inequality of \eqref{leibi}, we have
\begin{align*}
    \left(\frac{1-\cos R}{\sin R}\right)^2 \le \frac{A^2}{l^2} \le \frac{A^2}{A(A_{\infty}-A)}.
\end{align*}Hence 
\begin{align*}
    \frac{1-\cos R}{1+\cos R} \le \frac{A}{A_{\infty}-A}.
\end{align*}After simplification, we obtain \eqref{busanhuang}, and \eqref{bushanhuangtuichu} is just a direct consequence of \eqref{busanhuang}.
\hfill\\

It remains to show \eqref{diamradial}. This can be proved by using argument similar to that used in the proof of Proposition \ref{diametercontrolinfty}:

For any $x, y \in \mathbb{R}^n$, We let $\gamma_1, \, \gamma_2$ be the rays staring from $0$ and passing through $x$ and $y$ respectively. Let $\gamma_x$ be the line segment connecting $0$ and $x$, and let $\gamma^x=\gamma_1 \setminus \gamma_x$. Similarly, let $\gamma_y$ be the line segment connecting $0$ and $y$, and let $\gamma^y=\gamma_2 \setminus \gamma_y$. We proceed similarly as the proof of Theorem \ref{maintheoremonradialcase}:

If $l(\gamma_x)+l(\gamma_y) \le \pi$, then this already implies $d_g(x,y) \le \pi$ by triangular inequality.

If $l(\gamma_x)+l(\gamma_y) \ge \pi$, then since $l(\gamma_i) \le \pi, \, i=1,2$ as we proved in Proposition \ref{aiya}, we have
\begin{align}
\label{moshenyihao3'}
    l(\gamma^x)+l(\gamma^y) \le \pi.
\end{align}
Also, by \eqref{jiandanyouyong} which we already proved, we know that for any $\epsilon>0$, we can choose $x_R=\gamma^x \cap \partial B_R$ and $y_R=\gamma^y \cap \partial B_R$ such that $d_g(x_R, y_R)\le \int_{\partial B_R} e^u ds<\epsilon$. Therefore, by \eqref{moshenyihao3'},
\begin{align*}
    d_g(x,y)\le d_g(x,x_R)+d_g(x_R,y_R)+d_g(y_R,y) \le \pi+\epsilon.
\end{align*}
Let $\epsilon \rightarrow 0$, we have $d_g(x,y)\le \pi$. 
\end{proof}
\begin{remark}
\label{bushang} 
From the proof, one can see that if either of the inequalities in \eqref{leibi} becomes equalities for some value $r>0$, then $u$ must be a solution to \eqref{scalliouville} and then $(\mathbb{R}^2, e^{2u}\delta)$ is a sphere minus a point.
\end{remark}

\subsection*{Some alternative proofs of some of the inequalities in Theorem \ref{maintheoremonradialcase}}
\begin{itemize}
\item  \eqref{volumeradial} is proved as a consequence of the second inequality of \eqref{leibi}. Actually the first inequality of \eqref{leibi} also implies \eqref{volumeradial}, since 
    \begin{align*}
        A_{\infty}-A\le \frac{l^2}{A}.
    \end{align*}Let $r\rightarrow 0$ on both sides of the above inequality, we have $A_{\infty} \le 4\pi$, which is exactly \eqref{volumeradial}.
    
\item \eqref{jiefa} can also be proved by applying Alexandrov inequality. 
    \begin{proof}[Second proof of \eqref{jiefa}]
To show that $l(r) \ge A(r)$ for $r \le r_0$, we first exploit the following Alexandrov inequality (see \cite{Top99} for a proof): For any $K_0 \in \mathbb{R}$,
\begin{align}
\label{alexandrov'}
    4\pi A \le L^2 +K_0A^2+2A\int_{\Omega}(K-K_0)_+ dvol_g.
\end{align}Let $K_0=1$, and we apply \eqref{alexandrov'} to $(B_r, g)$. Note that \begin{align*}
    (K-1)_+=(-(\Delta u) e^{-2u}-1)_+=\left(e^{-2u}(-\Delta u-e^{2u})\right)_+=e^{-2u}(-\Delta u-e^{2u}),
\end{align*}hence we have
\begin{align*}
    4\pi A-A^2 \le & l^2+2A\int_{B_r} (-\Delta u-e^{2u}) dx\\
    =&l^2+2A\left(2\pi\int_0^r(-\rho u'(\rho))'d\rho-A\right)\\
    =&l^2+2A\left(-2\pi ru'(r)-A\right).
\end{align*}Hence
\begin{align*}
    l^2 \ge& 4\pi A(1+ru'(r)) +A^2\\
    \ge& A^2, \quad \mbox{when $r\le r_0$}
\end{align*}since $ru'(r) \ge -1$ when $r \le r_0$.

\end{proof}

\item \begin{proof}[Alternative proof of \eqref{yanxing3}:]
If $l(r)$ achieves its maximum at $r_0$, then $l'(r_0)=0$, and thus $1+r_0e^{u(r_0)} =0$. Hence $\partial B_{r_0}$ is a geodesic under the metric $g=e^{2u}$. Applying Toponogov's Theorem, we have $l(r_0) \le l_{S^2}(\partial B_{R(r_0)}')$, where $R(r)$ is the function defined above. Therefore,
\begin{align*}
    l(r_0) \le 2\pi \sin R(r_0) \le 2\pi.
\end{align*}
\end{proof}

\end{itemize}

In the end of this section, we remark that if $u$ satisfies \eqref{ricci1} and $(\mathbb{R}^2, e^{2u}\delta)$ can be completed as a closed Riemannian surface, then the first inequality of \eqref{leibi} is exactly the L\'evy-Gromov isoperimetric inequality, while the second inequality is equivalent to that of \cite[Corollary 3.2]{NiWang}, provided one can show that any minimizer to the functional
\begin{align*}
    h_{\beta}(\Omega):=\Big\{\frac{\int_{\partial \Omega}e^{u} ds}{\int_{\mathbb{R}^2}e^{2u}dx}:\, \Omega \subset \mathbb{R}^2,\quad  \frac{\int_{\Omega} e^{2u}dx}{\int_{\mathbb{R}^2}e^{2u}dx}=\beta\Big\}
\end{align*}must be a ball centered at the origin.

\section{Higher Dimensional Results for radial solutions to generalized Liouville equation}
There are a few directions to extend the study of solutions or supersolutions to Liouville equation \eqref{scalliouville} in $\mathbb{R}^2$ to higher dimensional spaces. For example, \eqref{scalliouville} can be viewed as the equation describing globally conformally flat Einstein manifolds with $Ric_g=(n-1)g$ where $n=2$ is the dimension of the manifold. In higher dimensional case, it is natural to consider corresponding solutions to equation $Ric_g=(n-1)g$, $g=e^{2u}\delta$ where $\delta$ is the classical metric in $\mathbb{R}^n$.

By \cite[Page 58]{Besse}, we have the formula for Ricci tensor
\begin{align}
\label{ricciconformalchangeformula}
    Ric_g=(2-n)(\nabla d u-du \otimes du)+(-\Delta u-(n-2)|\nabla u|^2) \delta.
\end{align}Therefore, the higher dimensional Liouville equation reads
\begin{align}
\label{higherdimensionalLiouville}
    (2-n)(\nabla d u-du \otimes du)+(-\Delta u-(n-2)|\nabla u|^2) \delta=(n-1)e^{2u}\delta.
\end{align}

Let $u$ be a radial solution to $Ric_g \ge (n-1)g$. We first recall the equation when $u$ is radially symmetric.

Let $h$ be the unit round metric on $S^{n-1}$, then 
\begin{align*}
    \delta=dr^2+r^2h
\end{align*}and thus
\begin{align}
\label{change}
    \delta_{ij}-\frac{x_ix_j}{r^2}=r^2h_{ij}
\end{align}where $h_{ij}=h(\partial_i,\partial_j)$.
Hence if $u$ is radially symmetric, then we have
\begin{align*}
    u_iu_j=u_r^2\frac{x_ix_j}{r^2}=u_r^2dr^2(\partial_i,\partial_j)
\end{align*}and 
\begin{align*}
    u_{ij}=&u_{rr}\frac{x_j}{r}\frac{x_i}{r}+\frac{u_r}{r}\left(\delta_{ij}-\frac{x_ix_j}{r^2}\right)\\
    =&[u_{rr}dr^2+ru_rh](\partial_i,\partial_j), \quad \mbox{by \eqref{change}}.
\end{align*}Since $du=\sum_iu_i dx_i$ and $\nabla du=\sum_{i,j}u_{ij}dx_i \otimes dx_j$, we have \begin{align}
    \label{hessiansphericalcoordinates}
\nabla du=u_{rr}dr^2+ru_rh,    
\end{align}and 
\begin{align}
    du\otimes du=u_r^2dr^2,
\end{align}
and thus \eqref{ricciconformalchangeformula} reads
\begin{align*}
    Ric_g=-(n-1)(u^{''}+\frac{u'}{r})dr^2-\left(u^{''}+(2n-3)\frac{u'}{r}+(n-2)(u')^2\right)r^2h.
\end{align*}Hence $Ric_g \ge (n-1)g$ is equivalent to
\begin{align}
    \label{higherdimensionalradiallycase}
\begin{cases}
u^{''}+\frac{u'}{r}+e^{2u}\le 0\\
u^{''}+(2n-3)\frac{u'}{r}+(n-2)(u')^2+(n-1)e^{2u}\le 0.
\end{cases}    
\end{align}

Motivated by Theorem \ref{maintheoremonradialcase}, it is natural to ask: If $(M,g)=(\mathbb{R}^n, e^{2u}\delta)$ where $u$ is a radial function satisfying \eqref{higherdimensionalradiallycase}, then is it true that $diam_g(\mathbb{R}^2) \le \pi$ and $vol_g(\mathbb{R}^n) \le vol_{g_{S^n}}(S^n)$?

The answer is yes. Actually, even assuming the weaker assumption \eqref{2dradialcase}, we still have the same diameter upper bound $\pi$. The proof of this is almost identical to the proof of \eqref{diamradial}.
\begin{proposition}
\label{zaoshuizaoqi}
If $u$ satisfies \eqref{2dradialcase}, and let $g=e^{2u(|\cdot|)}\delta$ where $\delta$ is the classical metric in $\mathbb{R}^n$, then $diam_g(\mathbb{R}^n) \le \pi$.
\end{proposition}
\begin{proof}
Let $x$ and $y$ be two points in $\mathbb{R}^n$, and let $\gamma_1, \, \gamma_2$ be the rays staring from $0$ and passing through $x$ and $y$ respectively. Let $\gamma_x$ be the line segment connecting $0$ and $x$, and let $\gamma^x=\gamma \setminus \gamma_x$. Similarly, let $\gamma_y$ be the line segment connecting $0$ and $y$, and let $\gamma^y=\gamma \setminus \gamma_y$. 

If $l(\gamma_x)+l(\gamma_y) \le \pi$, then this already implies $d_g(x,y) \le \pi$ by triangular inequality.

If $l(\gamma_x)+l(\gamma_y) \ge \pi$, then by Proposition \ref{aiya},
\begin{align}
\label{moshenyihao3}
    l(\gamma^x)+l(\gamma^y) \le \pi.
\end{align}
Since $\lim_{r \rightarrow \infty} 2\pi r e^{u(r)}=0$ as proved in Theorem \ref{maintheoremonradialcase}, we know that for any $\epsilon>0$, we can choose $x_R=\gamma^x \cap \partial B_R$ and $y_R=\gamma^y \cap \partial B_R$ such that $d_g(x_R, y_R)\le 2\pi R e^{u(R)}<\epsilon$. Therefore, by \eqref{moshenyihao3},
\begin{align*}
    d_g(x,y)\le d_g(x,x_R)+d_g(x_R,y_R)+d_g(y_R,y) \le \pi+\epsilon.
\end{align*}
Let $\epsilon \rightarrow 0$, we have $d_g(x,y)\le \pi$. 
\end{proof}

Now let us state the higher dimensional result for radially symmetric solutions.
\begin{proposition}
Let $u$ be the function satisfying \eqref{higherdimensionalradiallycase}, then under the metric $g=e^{2u(|\cdot|)}\delta$, where $\delta$ is the Euclidean metric in $\mathbb{R}^n$, then 
\begin{align}
\label{ndimdiameter}
    diam_g(\mathbb{R}^n) \le \pi
\end{align}
and \begin{align}
\label{ndimvolume}
    vol_g(\mathbb{R}^n) \le vol(S^n).
\end{align}
\end{proposition}
\begin{proof}
\eqref{ndimdiameter} is proved in Proposition \ref{zaoshuizaoqi}.

\eqref{ndimvolume} follows from Bishop-Gromov Theorem. Indeed, as discussed in the proof of \eqref{hmm}, Bishop-Gromov Theorem can be applied in $(\mathbb{R}^n, g)$ for $g$ to be a radially symmetric conformally flat metric.
\end{proof}
At the end of the section, we remark that any solution to \eqref{higherdimensionalLiouville} must be radially symmetric about a point. This is because any solution to \eqref{higherdimensionalLiouville} is also a solution to the Yamabe equation, and hence by \cite{CL91}, such solution is radially symmetric about a point.

\section{Connection with sphere covering inequality for the case $K \ge 1$}
In this section, we prove Theorem \ref{genghm} and Theorem \ref{radialspherecovering}.

First, let us state an equivalent version of L\'evy-Gromov isoperimetric inequality in closed Riemannian manifold of dimension $2$, which gives a simple algebraic relation between length and area.
\begin{proposition}
\label{xulie}
Let $M$ be a closed Riemannian manifold of dimension 2 and the Guassian curvature on $M$ is bounded below by $1$, then for any smooth domain $\Omega$ in $M$, we have
\begin{align}
\label{consequenceofgromov}
    l^2(\partial \Omega) \ge vol(\Omega) vol(M \setminus \Omega).
\end{align}The equality holds if and only if $M$ is a unit sphere and $\Omega$ is a spherical cap (geodesic disk).
\end{proposition}
\begin{proof}
Let $\beta=\frac{vol(\Omega)}{vol(M)}$, then by L\'evy-Gromov isoperimetric inequality, we have
\begin{align}
\label{levygromov}
    \frac{l(\partial \Omega)}{vol(M)}\ge \frac{l(\partial B_r)}{vol(S^2)},
\end{align}
where $S^2$ is the unit 2-sphere and $B_r$ is the geodesic ball in $S^2$ such that $\frac{vol(B_r)}{vol(S^2)}=\beta$. Hence $vol(B_r)=4\pi \beta$ and thus $l(\partial B_r)=4\pi \sqrt{\beta-\beta^2}$. Therefore, \eqref{levygromov} becomes
\begin{align*}
    \frac{l^2(\partial \Omega)}{vol^2(M)}\ge \beta(1-\beta)=\frac{vol(\Omega)}{vol(M)}\left(1-\frac{vol(\Omega)}{vol(M)}\right).
\end{align*}This implies \eqref{consequenceofgromov}. Since from the original proof of \eqref{levygromov}, it becomes equality for some $r>0$ if and only if $M$ is a sphere, we prove the equality case of \eqref{consequenceofgromov}.
\end{proof}

Now we prove Theorem \ref{genghm}.
\begin{proof}[Proof of Theorem \ref{genghm}]
Let $\lambda(t)=h(t)e^{-2t}$, hence \eqref{mc1''} implies $\lambda'(t) \le 0$. 
Set $\alpha(t)=\int_{\{u>t\}}\lambda(u)e^{2u}d\mu$ and $\beta(t)=\mu\left(\{u>t\}\right)$. Hence 
\begin{align*}
    \beta'(t)=-\int_{\{u=t\}}\frac{1}{|\nabla u|} ds \quad \mbox{and $\alpha'(t)=\lambda(t)e^{2t}\beta'(t)$.}
\end{align*}
We integrate \eqref{rubin2''} over $\{u>t\}$, and by divergence theorem we have \begin{align*}
    \int_{\{u=t\}}|\nabla u|ds +\beta \le \alpha.
\end{align*}
Hence \begin{align*}
    -\beta'(\alpha-\beta) \ge \left(\int_{\{u=t\}}|\nabla u|ds\right)\left(\int_{\{u=t\}}\frac{1}{|\nabla u|}ds\right) \ge l^2(\partial \{u>t\}),
\end{align*}where $l$ is the length function on $M$.
We multiply the above by $\lambda(t) e^{2t}$, and using $\alpha'=\lambda e^{2t} \beta'$ and applying \eqref{consequenceofgromov}, we have 
\begin{align*}
    -\alpha \alpha' +\alpha' \beta \ge& \lambda e^{2t} (\mu(M) \beta-\beta^2)\\
    =&\lambda e^{2t}\mu(M) \beta-\frac{(\lambda e^{2t}\beta^2)'-\lambda'e^{2t}\beta^2-2\beta \alpha'}{2}.
\end{align*}Hence \begin{align}
\label{1cru}
    -2\alpha \alpha'\ge 2\mu(M) \lambda e^{2t}\beta -(\lambda e^{2t}\beta^2)'+\lambda'e^{2t} \beta^2.
\end{align}
Note that
\begin{align*}
    2\lambda e^{2t}\beta=(\lambda e^{2t}\beta)'-\lambda' e^{2t} \beta -\alpha',
\end{align*}where we have again used that $\alpha'=\lambda e^{2t}\beta'$.
Hence \eqref{1cru} becomes 
\begin{align}
    \label{2cru}
-2\alpha \alpha' \ge \left(\lambda e^{2t}\beta (\mu(M)-\beta)\right)'-\mu(M)\alpha'+\lambda'e^{2t}(\beta^2-\mu(M)\beta).    
\end{align}
Since $\lambda'\le 0$ and $\beta^2-\mu(M) \beta  \le 0$, we have 
\begin{align}
\label{jianhua}
    -2\alpha \alpha' \ge \left(\lambda e^{2t}\beta (\mu(M)-\beta)\right)'-\mu(M)\alpha'.
\end{align}
Then we integrate \eqref{jianhua}  from $t=0$ to $\infty$, and using the fact that $\lim_{t \rightarrow \infty}\alpha(t)=0$ and $\lim_{t \rightarrow \infty} \lambda e^{2t}\beta (\mu(M)-\beta) \le \mu(M) \lim_{t\rightarrow \infty}\alpha(t)=0$, we have 
\begin{align*}
    \mu(M) \alpha(0) -\alpha(0)^2 \le \lambda(0)\beta(0)\left(\mu(M)-\beta(0)\right).
\end{align*}
That is,
\begin{align*}
    \mu(M)\int_{\Omega}h(u) d\mu-\left(\int_{\Omega} h(u) d\mu\right)^2 \le h(0)\left( \mu(M) \mu(\Omega)-\mu^2(\Omega)\right).
\end{align*} 
In particular, if $h(0)\le 1$, then\begin{align*}
    \mu(M) \int_{\Omega}h(u) d\mu-\left(\int_{\Omega} h(u) d\mu\right)^2 \le \mu(M) h(0) \mu(\Omega)-h^2(0)\mu^2(\Omega).
\end{align*}Hence
\begin{align*}
    \mu(M) \left(\int_{\Omega} h(u) d\mu -h(0) \mu(\Omega)\right) \le \left(\int_{\Omega} h(u) d\mu -h(0) \mu(\Omega)\right) \left(\int_{\Omega} h(u) d\mu +h(0) \mu(\Omega)\right).
\end{align*}Hence 
\begin{align*}
    \int_{\Omega}h(u) d\mu +h(0)\mu(\Omega) \ge \mu(M).
\end{align*}
The proof is complete. 
\end{proof}

Theorem \ref{genghm} also has its dual form:
\begin{theorem}
\label{genghm'}
Let $(M,g)$ be a closed Riemannian surface with and $K \ge 1$, where $\mu$ is the measure of $(M,g)$ and $K$ is the Gaussian curvature. Let $\Omega$ be a domain with compact closure and nonempty boundary. If $u$ satisfies 
\begin{align}
\label{rubin2'}
  \begin{cases}
-\Delta_g u+1 \ge h(u), \quad  u<0 \quad &\mbox{in $\Omega$}\\
u=0 \quad &\mbox{on $\partial \Omega$}
\end{cases} 
\end{align}where $h(t)$ is a nonnnegative function satisfying \eqref{mc1''}, i.e., \begin{align}
\label{mc1-new}
   0 \le  h'(t) \le 2 h(t),
\end{align}
then \begin{align}
\label{guijia1}
    \mu(M) \int_{\Omega}h(u) d\mu -\left(\int_{\Omega} h(u) d\mu\right)^2 \le h(0) \mu(\Omega) \mu (M\setminus \Omega).
\end{align}
Moreover, if $h_0 \ge 1$, then 
\begin{align}
\label{guijia2}
    \int_{\Omega}h(u) d\mu +h(0)\mu(\Omega) \ge \mu(M).
\end{align}
\end{theorem}
\begin{proof}
We still let $\lambda(t)=h(t)e^{-2t}$, $\alpha(t)=\int_{\{u<t\}}\lambda(u)e^{2u}d\mu$ and $\beta(t)=\int_{\{u<t\}}d\mu$. Then $\beta'(t)=\int_{\{u=t\}}\frac{1}{|\nabla u|}ds$ and $\alpha'(t)=\lambda(t)e^{2t}\beta'(t)$.
We integrate \eqref{rubin2'} over $\{u<t\}$ and by divergence theorem, we have
\begin{align*}
    \int_{\{u=t\}}|\nabla u| ds -\beta \le -\alpha.
\end{align*}Hence 
\begin{align*}
    \beta'(\beta-\alpha)\le\left( \int_{\{u=t\}}|\nabla u|ds \right)\left(\int_{\{u=t\}}\frac{1}{|\nabla u|}ds\right)\ge s^2(\partial \{u<t\}).
\end{align*}
Multiply the above by $\lambda(t)e^{2t}$, using $\alpha'(t)=\lambda(t)e^{2t}\beta'(t)$ and the isoperimetric inequality, we have
\begin{align*}
    \alpha'(\beta-\alpha) \ge \lambda e^{2t} (\mu(M) \beta -\beta^2).
\end{align*}
Then arguing simiarly as the proof of Theorem \ref{genghm}, we still get \eqref{jianhua}, that is, \begin{align*}
    -2\alpha \alpha' \ge \left(\lambda e^{2t}\beta (\mu(M)-\beta)\right)'-\mu(M)\alpha'.
\end{align*}. 
Then we integrate the above inequality from $-\infty$ to $0$ and thus obtain
\begin{align*}
    \mu(M) \alpha(0)-\alpha^2(0) \ge \lambda(0)\beta(0)\left(\mu(M)-\beta(0)\right).
\end{align*}
That is, \begin{align*}
    \mu(M) \int_{\Omega}h(u) d\mu -\left(\int_{\Omega} h(u) d\mu\right)^2 \ge h(0)\left(\mu(M) \mu(\Omega)-\mu^2(\Omega)\right).
\end{align*}
If $h(0)\ge 1$, then 
\begin{align*}
     \mu(M) \int_{\Omega}h(u) d\mu-\left(\int_{\Omega} h(u) d\mu\right)^2 \ge \mu(M) h(0) \mu(\Omega)-h^2(0)\mu^2(\Omega).
\end{align*}
Hence
\begin{align*}
    \mu(M) \left(\int_{\Omega} h(u) d\mu -h(0) \mu(\Omega)\right) \ge \left(\int_{\Omega} h(u) d\mu -h(0) \mu(\Omega)\right) \left(\int_{\Omega} h(u) d\mu +h(0) \mu(\Omega)\right).
\end{align*}Since $u<0$ in $\Omega$, $h(u)<h(0)$ and thus $\int_{\Omega} h(u)d\mu -h(0)\mu(\Omega)\le 0$. Therefore, \begin{align*}
    \int_{\Omega}h(u) d\mu +h(0)\mu(\Omega) \ge \mu(M).
\end{align*}
This completes the proof.
\end{proof}

\begin{remark}
If the  Gaussian curvature $K$ of  $(M, g)$ satisfies   $a^2\le K\le 1$ for some positive
constant $a<1$, and $u$ satisfies conditions of Theorem \ref{genghm} with $h(0) \le a^2$,
then the conclusions of Theorem \ref{genghm} still hold with $\mu, h(u), \lambda$ replaced by $\tilde \mu= a^2 \mu, \tilde h (u) := h(u)/a^2,  \tilde \lambda=\lambda/a^2$ respectively, after applying Theorem \ref{genghm} with a proper scaling of the metric $g$ by $ag$.  In particular, if $\lambda\le a^2$, then  \eqref{guijia2''}  becomes 
\begin{align}
    \label{guijia2-new}
\int_{\Omega}e^{2u} d\mu +\mu(\Omega) \ge \frac{a^2\mu(M)}{\lambda}.    
\end{align}
It is interesting to compare this with the lower bound
obtained from \eqref{linshiruji} of Theorem \ref{ghm1.4},  where $\lambda$ is allowed to be in $(0,1]$.  Note that $ a^2 \mu(M) \le 4 \pi$ by the Gauss-Bonnet theorem.  It turns out that under the same curvature conditions $a^2\le K\le 1$,   \eqref{linshiruji} in Theorem \ref{ghm1.4} requires less constraints and has a better lower bound.  Nevertheless,   Theorem \ref{genghm} gives a similar lower bound but under a complete opposite curvature condition.   Similar remark can also be made for Theorem \ref{genghm'}, the dual form of Theorem \ref{genghm}.

\end{remark}
Next, we prove Theorem \ref{radialspherecovering}.
\begin{proof}[Proof of Theorem \ref{radialspherecovering}]
First, let $u=u_2-u_1$, and hence $u>0$ in $B_r$ and $u=0$ on $\partial B_r$. Note that for any $t>0$, $\{u>t\}$ is a radially symmetric domain. 

We claim that if $\omega$ is a radially symmetric domain, then \begin{align}
\label{liantongargument}
    P^2(\omega) \ge A(\omega)(A_{\infty}-A(\omega)),
\end{align}
where $P(\omega)=\int_{\partial \omega}e^{u_1}ds, A(\omega)=\int_{\omega}e^{2u_1}dx$ and $A_{\infty}=\int_{\mathbb{R}^2}e^{2u_1}dx$.

Indeed, if $\omega$ is connected, then $\omega$ is either a disk or an annulus centered at origin. For the disk case, \eqref{liantongargument} is exactly the first inequality of \eqref{leibi}. If $\omega$ is an annulus, then $\omega=\omega_1 \setminus \omega_2$, where $\omega_1 \Supset \omega_2$ are two disks. Then by \eqref{leibi}, we have
\begin{align*}
    P^2(\omega) - A(\omega)(A_{\infty}-A(\omega))\ge& P^2(\omega_1)+P^2(\omega_2)-\left(A(\omega_1)-A(\omega_2)\right)\left(A_{\infty}-A(\omega_1)+A(\omega_2)\right)\\\ge&\sum_{i=1}^2A(\omega_i)(A_{\infty}-A(\omega_i))-A_{\infty}\left(A(\omega_1)-A(\omega_2)\right)+(A(\omega_1)-A(\omega_2))^2\\
    =&2A_{\infty}A(\omega_2)-2A(\omega_1)A(\omega_2)\ge 0.
\end{align*}
Hence we proved \eqref{liantongargument} for connected radially symmetric domains. If $\omega$ has more than $1$ component, then set $\omega=\cup_i \omega_i$, where $\omega_i$ are disjoint component. Hence $\omega_i$ satisfies \eqref{liantongargument}. We have
\begin{align*}
    P^2(\omega) - A(\omega)(A_{\infty}-A(\omega))\ge & \sum_i P^2(\omega_i)-\left(\sum_i A(\omega_i)\right)(A_{\infty}-\sum_iA(\omega_i))\\
    \ge& \sum_iA(\omega_i)(A_{\infty}-A(\omega_i))-\left(\sum_i A(\omega_i)\right)(A_{\infty}-\sum_iA(\omega_i))\\
    =&\left(\sum_i A(\omega_i)\right)^2-\sum_iA^2(\omega_i)\ge 0.
\end{align*}Hence \eqref{liantongargument} holds for general radially symmetric domains.

Since $u$ satisfies \eqref{rubin2''} for $g=e^{2u_1}\delta$, applying \eqref{liantongargument} for the radially symmetric domain $\{u>t\}$ instead of \eqref{consequenceofgromov} in the proof of Theorem \ref{genghm}, and proceeding the same argument, we have
\begin{align*}
    \int_{B_r}e^{2u}d\mu+\mu(B_r)\ge \mu(\mathbb{R}^2).
\end{align*}Since $\mu=e^{2u_1}dx$ and $u_2=u_1+u$, the above inequality becomes
\begin{align*}
    \int_{B_r}e^{2u_2} dx+\int_{B_r}e^{2u_1} dx\ge \int_{\mathbb{R}^2}e^{2u_1}dx.
\end{align*}When the equality holds, then by tracing the equality cases, especially by Remark \ref{bushang}, we conclude that $(\mathbb{R}^2, e^{2u_1}\delta)$ is punctured sphere. In such case, we have \begin{align*}
     \Delta u_2+e^{2u_2}= \Delta u_1+e^{2u_1}=0 \, \mbox{in $B_r$}, \, u_2>u_1 \, \mbox{in $B_r$ and} \, u_2=u_1 \, \mbox{on $\partial B_r$,}
\end{align*}Hence as shown in \cite{GM}, $(B_r, e^{2u_1}\delta)$ and $(B_r,e^{2u_2}\delta)$ are two complementary spherical caps on the unit sphere. 
\end{proof}

The proof of Theorem \ref{radialspherecovering'} is similar because of Theorem \ref{genghm'}, so we omit it.

\section{Some unsolved problems and remarks}
In this section, we present some unsolved problems related to the results proved in this paper. Much further research can be done as a continuation of this paper. 

The following question is on conformal diameter corresponding to solutions to \eqref{scalliouville}.
\begin{question} Let $u$ be a solution to \eqref{scalliouville} and $g=e^{2u}\delta$, where $\delta$ is the Euclidean metric.  Is it true that there exists a universal constant $C\ge 20$ such that \begin{align*}
    \pi \le diam_g(\mathbb{R}^2) \le C\pi?
\end{align*}Moreover, is it true that $diam_g(\mathbb{R}^2)=\pi$ if and only if $u$ is radial up to translation, or $u$ is an $1$D solution to \eqref{scalliouville}? 

\end{question}
    
We remark that in our forthcoming paper \cite{EGLX}, we can prove if $u$ has an upper bound, then under the metric $e^{2u}\delta$, diameter of $\mathbb{R}^2$ equal to $\pi$ if and only if $u$ is either radial about a point, or $u$ is one-dimensional.

The following question is on conformal diameter corresponding to supersolutions to \eqref{scalliouville}.

\begin{question}
On general supersolutions, let $u$ satisfy \eqref{ricci1} and $g=e^{2u}\delta$, where $\delta$ is the Euclidean metric. Suppose that $\int_{\mathbb{R}^2} e^{2u} dx <\infty$, is it true that $diam_g(\mathbb{R}^2) \le \pi$? Recall that in Proposition \ref{weakbounddiamter}, we can only prove the $2\pi$ upper bound.
\end{question}

The following question is on conformal volume estimate coorespnding to supersolutions to \eqref{scalliouville}.
\begin{question}
\label{volume4pibound}
let $u$ satisfy \eqref{ricci1} and $g=e^{2u}\delta$, where $\delta$ is the Euclidean metric. Suppose that $\int_{\mathbb{R}^2} e^{2u} dx <\infty$, then is it true that \begin{align}
    \label{buzifa}
\int_{\mathbb{R}^2} e^{2u}dx \le 4\pi?    
\end{align}
\end{question}
Note that if the completion of $(\mathbb{R}^2, e^{2u}\delta)$ is a closed surface, then by Gauss-Bonnet theorem, \eqref{buzifa} is true. Let $K(x)$ be the Gaussian curvature at $x$, then a necessary condition for this to be true is that
\begin{align}
    \label{totalgausscurvure}
\int_{\mathbb{R}^2}K(x)e^{2u(x)}=4\pi, 
\end{align}where $K$ is the Gaussian curvature, $K=-e^{-2u}\Delta u$.  However, generally this is not true even for radial supersolutions to \eqref{scalliouville}, even if the metric can be smoothly extended at $\infty$, as we shall see immediately in the next example:
\begin{example}
\label{noncomplete}
    Let $u$ be the radial function $u=-r^2$, then $(\mathbb{R}^2, e^{2u}\delta)$ is a Riemannian manifold with Gaussian curvature $K=-\Delta u e^{-2u}=4e^{2r^2} \ge 1$. Using the conformal change of variable $z\mapsto \frac{1}{z}$, the metric at $\infty$ is equivalent to $h(r):=e^{2u(1/r)}\frac{1}{r^4}dr^2$ at $0$. Clearly, $h(r)$ is smooth at $0$, but 
    \begin{align*}
        \int_{\mathbb{R}^2} K dvol_g=\int_{\mathbb{R}^2}(-\Delta u)dx=\infty \ne 4\pi.
    \end{align*}
\end{example}
From the example above, we can see that generally one cannot directly apply geometric results on closed Riemannian manifolds to study supersolutions to the Liouville equation even for radial cases. 

Recall that our proof of \eqref{buzifa} is based on the geometric inequality \eqref{leibi}. So far we don't know the answer to Question \ref{volume4pibound} except the radial cases, however, if the inequality \eqref{ricci1} goes in the other direction, we do have the following counterpart for general cases:
\begin{proposition}
\label{Kle1volumeestimate}
If \begin{align}
\label{Kle1liouvile}
    \Delta u+e^{2u}\ge 0,
\end{align}then 
\begin{align}
\label{inversedirection}
    \int_{\mathbb{R}^2} e^{2u} dx \ge 4\pi.
\end{align}
\end{proposition}
The proof of the proposition follows exactly from the proof of \cite[Lemma 1.1]{CL91}. As a convenience for readers, we include the proof here. 
\begin{proof}
It suffices to prove \eqref{inversedirection} by assuming $\int_{\mathbb{R}^2} e^{2u} dx<\infty$.

Let $\Omega_t$ be the superlevel set of $u$. Since $|\Omega_t|<\infty$ because of Chebyshev's inequality, we have
\begin{align}
    \int_{\Omega_t}e^{2u} dx \ge -\int_{\Omega_t}\Delta u dx=\int_{\partial \Omega_t}|\nabla u|  dx
\end{align} and
\begin{align}
\label{jing}
    -\frac{d}{dt}|\Omega_t|=\int_{\partial \Omega_t} \frac{1}{|\nabla u|}.
\end{align}
Using
\begin{align*}
    \left(\int_{\partial \Omega_t}\frac{1}{|\nabla u|}dx\right) \left(\int_{\partial \Omega_t} |\nabla u| dx \right) \ge |\partial \Omega_t|^2 \ge 4\pi|\Omega_t|,
\end{align*}
we have
\begin{align*}
    -\frac{d}{dt}|\Omega_t| \int_{\Omega_t}e^{2u} dx \ge 4\pi |\Omega_t|.
\end{align*}
Hence
\begin{align*}
    \frac{d}{dt}\left( \int_{\Omega_t} e^{2u} dx \right)^2=2e^{2t} (\frac{d}{dt}|\Omega_t| )\int_{\Omega_t} e^{2u}dx \le -8\pi e^{2t} |\Omega_t|.
\end{align*}
Integrating from $-\infty$ to $\infty$, and since 
\begin{align*}
    \int_{-\infty}^{\infty}2e^{2t}|\Omega_t|dt=\int_{\mathbb{R}^2}e^{2u} dx,
\end{align*}
we proved \eqref{inversedirection}.
\end{proof}

\begin{question}
On radial supersolutions, recall that we have applied several geometric argument to prove the inequalities in Theorem \ref{maintheoremonradialcase}. Can we give a proof from pure analytic point of view?
\end{question}



The following question is on generalization of Proposition \ref{xulie} to the case $M=(\mathbb{R}^2, e^{2u}\delta)$ such that $u$ is a supersoluton to \eqref{scalliouville} and that $vol_g(M)<\infty$. 

\begin{question}
If $u$ satisfies \eqref{ricci1} and $\int_{\mathbb{R}^2}e^{2u} dx <\infty$, then for any smooth domain $\Omega \subset \mathbb{R}^2$, is it true that
\begin{align}
\label{conformalversionoflevygromov}
    \left(\int_{\partial \Omega} e^{u} ds\right)^2 \ge \left(\int_{\Omega} e^{2u} dx\right) \left(\int_{\mathbb{R}^2\setminus \Omega} e^{2u} dx\right)?
\end{align}
\end{question}
We note that \eqref{conformalversionoflevygromov} already implies \eqref{buzifa}.

Also, if the answer to this question is positive, then by similar argument in the proof of Theorem \ref{radialspherecovering}, we can extend Theorem \ref{radialspherecovering} to the case that $u_1, u_2$ are not necessarily radial and $\Omega$ can be a general smooth domain.


\vskip0.5in

$\mathbf{Acknowledgement}$  
This research is partially supported by  NSF grants DMS-1601885 and DMS-1901914 and and  Simons Foundation  Award 617072.


\begin{thebibliography}{10}
\bibitem{Bayle} V. Bayle. A differential inequality for the isoperimetric profile. International Math. Res. Not., 7:311-342, 2004.

\bibitem{Besse} A.L. Besse.  Einstein manifolds, Ergeb. Math. Grenzgeb. Band 10, Springer, Berlin, 1981.

\bibitem{CY87} ] A. Chang and P. Yang. Prescribing Gaussian curvature on $S^2$. Acta Math. 159 (1987), no. 3-4, 215–259. doi:10.1007/BF02392560.

\bibitem{CL91} W. Chen and C. Li.  Classification of solutions of some nonlinear elliptic equations, Duke Math. J., 63 (1991) 615–622.

\bibitem{Cheng75} S.Y. Chen(1975). Eigenvalue comparison theorems and its geometric applications. Mathematische Zeitschrift, 143 (3): 289–297.

\bibitem{CW94} K.Chou and T. Wan. Assymptotic radial symmetry of solution to $\Delta u+e^u=0$ in a punctured disk. Pacific Journal Of Mathematics. Vol. 163, No. 2, 1994.

\bibitem{Du}  Z. Du, C. Gui, J.  Jin and Y. Li. Multiple axially asymmetric solutions  to  a mean field equation on $\mathbb{S}^{2}$, Vol. 36, No. 1 (2020), pp. 19-32, DOI: 10.4208/ata.OA-0016.

\bibitem{EGLX} A. Eremenio, C. Gui, Q. Li and L. Xu. Rigidity results on Liouville equation. Preprint.

\bibitem{GLH} S. Gallor, J. Lafontaine and D. Hulin. Riemannian Geometry. Universitext. Springer-Verlag, 2nd edition, 1990.

\bibitem{GilbargTrudinger} D. Gilbarg and N. Trudinger. Elliptic partial differential equations of second order. Classics in Mathematics. Springer. 

\bibitem{Gromov} M. Gromov. Metric structures for Riemannian and non-Riemannian spaces. With appendices by M. Katz, P. Pansu and S. Semmes, volume 152 of Progress in Mathematics. Birkh\"auser, 2nd edition, 1999.

\bibitem{25} G. Gu, C. Gui, Y. Hu and Q. Li. Uniqueness and symmetry results of a mean field equation on arbitrary flat tori, accepted by Int. Math. Res. Not. IMRN .

\bibitem{26} C. Gui, F. Hang, A. Moradifam and X. Wang. Remarks on a mean field equation on $S^2$. J. Math. Study, in press. Preprint, 2019. arXiv:1905.10842.

\bibitem{GHM} C. Gui, F. Hang and A. Moradifam. The sphere covering inequality and its dual. Communications on Pure and Applied Mathematics, Vol. LXXIII, 0001–0023 (2020).

\bibitem{27} C. Gui, A. Jevnikar and A. Moradifam. Symmetry and uniqueness of solutions to some Liouvilletype equations and systems. Comm. Partial Differential Equations 43(2018), no. 3, 428–447. doi:10.1080/03605302.2018.1446164.

\bibitem{GM} C. Gui and  A. Moradifam.  The sphere covering inequality and its applications. Invent. Math. 214
(2018), no. 3, 1169–1204. doi:10.1007/s00222-018-0820-2

\bibitem{29} C. Gui and A. Moradifam. Uniqueness of solutions of mean field equations in $\mathbb{R}^2$. Proc. Amer. Math. Soc. 146 (2018), no. 3, 1231–1242. doi:10.1090/proc/13814.

\bibitem{30} C. Gui and A. Moradifam. Symmetry of solutions of a mean field equation on flat tori. Int. Math.
Res. Not. IMRN (2019), no. 3, 799–809. doi:10.1093/imrn/rnx121.

\bibitem{31} Y. Lee, C.S. Lin, G. Tarantello and W. Yang. Sharp estimates for solutions of mean field equations with collapsing singularity. Comm. Partial Differential Equations 42 (2017), no. 10, 1549-1597. doi:10.1080/03605302.2017.1382519.

\bibitem{HK78} E. Heintze and H. Karcher, A general comparison theorem with applications to volume estimates for submanifolds, Ann. Sci. Ec. Norm. Super. 11 (1978), 451-470. 

\bibitem{Laine92} Laine, Ilpo. (1992). Nevanlinna Theory and Complex Differential Equation, (Berlin-New York: Walter de Gruyter). Walter de Gruyter. 10.1515/9783110863147. 

\bibitem{Lee} J. Lee. Reimannian manifolds : an introduction to curvature.  Graduate texts in mathematics 176.

\bibitem{Liouville}J. Liouville.  Sur Vaquation aux Derives Partielles $\partial^2 log \lambda /\partial u \partial v \pm 2\lambda a^2=0$. de Math., 18 (1), (1853), 71-72.

\bibitem{NiWang} L. Ni and K. Wang. Isoperimetric Comparisons via Viscosity. J Geom Anal. DOI 10.1007/s12220-015-9650-2.


\bibitem{39} Y. Shi, J. Sun, G. Tian and D. Wei. Uniqueness of the mean field equation and rigidity of Hawking mass. Calc. Var. Partial Differential Equations 58 (2019), no. 2, Paper no. 41, 16 pp. doi:10.1007/s00526-019-1496.


\bibitem{Top98} P.M. Topping. Mean Curvature Flow and Geometric Inequalities. J. reine angew. Math. 503, 47–61 (1998)

\bibitem{Top99} P.M.Topping. The isoperimetric inequality on a surface. Manuscripta Mathematica volume 100, pages23-33(1999).

\bibitem{41} J. Wang, Z.Wang and W. Yang. Uniqueness and convergence on equilibria of the Keller-Segel
system with subcritical mass. Comm. Partial Differential Equations 44 (2019), no. 7, 545–572. doi:10.1080/03605302.2019.1581804.

\bibitem{Yanglo} L. Yang. Value Distribution Theory. Springer, Berlin, Germany, 1993.


 
\end{thebibliography}
\end{document}